\documentclass[reqno]{amsart}
\usepackage[utf8]{inputenc}
\usepackage[letterpaper, portrait, margin=1.2in]{geometry}
\usepackage[english]{babel}
\usepackage{amsthm}
\usepackage[utf8]{inputenc}
\usepackage{graphicx}
\usepackage{mathtools}
\usepackage{amsmath}
\usepackage{amssymb}
\usepackage{tabto}
\usepackage{enumitem}
\usepackage{tikz}
\usepackage{hyperref}
\usepackage{cite}
\usepackage{subcaption}

\usetikzlibrary{calc}
\newtheorem*{theorem*}{Theorem}
\newtheorem{theorem}{Theorem}[section]
\newtheorem{corollary}[theorem]{Corollary}
\newtheorem{claim}[theorem]{Claim}
\newtheorem{problem}[theorem]{Problem}

\newtheorem{observation}[theorem]{Observation}
\newtheorem{remark}[theorem]{Remark}

\theoremstyle{definition}
\newtheorem{definition}[theorem]{Definition}

\DeclareMathOperator{\rev}{\mathrm{rev}}
\DeclareMathOperator{\Id}{\mathrm{Id}}
\DeclareMathOperator{\proj}{\mathrm{proj}}
\usepackage[skip=0.5\baselineskip]{caption}
\usepackage[toc,page]{appendix}

\title{On $\MakeLowercase{d}$-permutations and pattern avoidance classes}
\author{Nathan Sun} \address{Department of Applied Mathematics, Harvard University, Cambridge, MA 02138} \email{nsun@college.harvard.edu}
\date{April 21, 2024}

\begin{document}
\maketitle

\begin{abstract}
Multidimensional permutations, or $d$-permutations, are represented by their diagrams on $[n]^d$ such that there exists exactly one point per hyperplane $x_i$ that satisfies $x_i= j$ for $i \in [d]$ and $j \in [n]$. Bonichon and Morel previously enumerated $3$-permutations avoiding small patterns, and we extend their results by first proving four conjectures, which exhaustively enumerate $3$-permutations avoiding any two fixed patterns of size $3$. We further provide a enumerative result relating $3$-permutation avoidance classes with their respective recurrence relations. In particular, we show a recurrence relation for $3$-permutations avoiding the patterns $132$ and $213$, which contributes a new sequence to the OEIS database. We then extend our results to completely enumerate $3$-permutations avoiding three patterns of size $3$.

\end{abstract}

\section{Introduction}
Starting with Knuth's \cite{knuth1973art} work on permutations in 1973, the field of pattern avoidance has been well-studied in enumerative combinatorics. Simion and Schmidt first considered pattern avoidance in their work on enumerating permutation avoidance classes in 1985 \cite{simion1985restricted}. Pattern avoidance can be defined as follows:
\begin{definition}
Let $\sigma \in S_{n}$ and $\pi \in S_{k}$, where $k \leq n$. We say that the permutation $\sigma$ \emph{contains} the pattern $\pi$ if there exists indices $c_1 < \dots < c_k$ such that $\sigma(c_1)  \cdots \sigma(c_k)$ is order-isomorphic to $\pi$. We say a permutation \emph{avoids} a pattern if it does not contain it.
\end{definition}

Permutations avoiding sets of small patterns have been exhaustively enumerated \cite{simion1985restricted, mansour2020enumeration, knuth1973art}. It is well-known that permutations avoiding certain patterns are in bijection with other combinatorial objects, such as Dyck paths \cite{krattenthaler2001permutations, reifegerste2003diagram} and maximal chains of lattices \cite{simion1985restricted}. Some of them are further enumerated by the Catalan and Schr{\"o}der numbers \cite{west1995generating}. In their work, Simion and Schmidt \cite{simion1985restricted} completely enumerated permutations avoiding any single pattern, two patterns, or three patterns of size $3$, paving the path for more work in the field of pattern avoidance.

More recently, Bonichon and Morel \cite{bonichon2022baxter} considered a multidimensional generalization of a permutation, called a $d$-permutation, which resembles the structure of a $(d-1)$-tuple of permutations. Tuples of permutations have been studied before \cite{gunby2019asymptotics, aldred2005permuting}, but $d$-permutations have not been thoroughly studied yet, mainly appearing in a few papers related to separable permutations \cite{gunby2019asymptotics, asinowski2010separable}. In particular, Asinowski and Mansour \cite{asinowski2010separable} presented a generalization of separable permutations that are similar to $d$-permutations and characterized these generalized permutations with sets of forbidden patterns. The $d$-permutations studied by Bonichon and Morel coincide with the one introduced by Asinowski and Mansour \cite{asinowski2010separable} for the multidimensional case but also coincide with the classical permutation for $d=2$.

Similar to the enumeration Simion and Schmidt \cite{simion1985restricted} did in 1985 and Mansour \cite{mansour2020enumeration} in 2020, Bonichon and Morel \cite{bonichon2022baxter} started the enumeration of $d$-permutations avoiding small patterns and made many conjectures regarding the enumeration of $3$-permutations avoiding sets of two patterns. We present two main classes of results regarding the enumeration of $3$-permutation avoiding small patterns. We first completely enumerate $3$-permutations avoiding classes of two patterns of size $3$ and prove their respective recurrence relations, solving the conjectures presented by Bonichon and Morel \cite{bonichon2022baxter}. Further, we derive a recurrence relation for $3$-permutations avoiding $132$ and $213$, whose sequence we added to the OEIS database \cite{oeis}, and Bonichon and Morel did not provide any conjecture. We then further initiate and completely enumerate $3$-permutations avoiding classes of three patterns of size $3$, similar to Simion and Schmidt's results in 1985 \cite{simion1985restricted}. 

This paper is organized as follows. In Section 2, we introduce preliminary definitions and notation. In Section 3, we completely enumerate sequences of $3$-permutations avoiding two patterns of size 3 and prove four conjectures of Bonichon and Morel \cite{bonichon2022baxter}. In addition, we prove a recurrence relation for an avoidance class whose sequence we added to the OEIS database \cite{oeis}, completing our enumeration. In Section 4, we extend our enumeration to $3$-permutations avoiding three patterns of size 3 and prove recurrence relations for their avoidance classes. We conclude with open problems in Section 5.

\section{Preliminaries} \label{sec:preliminaries}

Let $S_n$ denote the set of permutations of $[n] = \{ 1,2, \dots, n \}$. Note that we can represent each permutation $\sigma \in S_n$ as a sequence $a_1 \cdots a_n$. Further, let $\mathrm{Id}_n$ denote the identity permutation $12 \cdots n$ of size $n$ and given a permutation $\sigma = a_1 \cdots a_n \in S_n$, let $\rev(\sigma)$ denote the reverse permutation $a_n \cdots a_1$. We further say that a sequence $a_1 \cdots a_n$ is \emph{consecutively increasing} (respectively \emph{decreasing}) if for every index $i$, $a_{i+1} = a_i+1$ (respectively $a_{i+1} = a_i-1$). 

For a sequence $a = a_1 \cdots a_n$ with distinct real values, the \emph{standardization} of $a$ is the unique permutation of $[n]$ with the same relative order. Note that once standardized, a consecutively-increasing sequence is the identity permutation and a consecutively-decreasing sequence is the reverse identity permutation. Moreover, we say that in a permutation $\sigma = a_1 \cdots a_n$, the elements $a_i$ and $a_{i+1}$ are \emph{adjacent} to each other. More specifically, $a_i$ is \emph{left-adjacent} to $a_{i+1}$ and similarly, the element $a_{i+1}$ is $\emph{right-adjacent}$ to $a_i$. The following definitions in this section were introduced in \cite{bonichon2022baxter}.

\begin{definition}
A \emph{$d$-permutation} $\boldsymbol{\sigma} := (\sigma_1 , \dots, \sigma_{d-1})$ of size $n$ is a tuple of permutations, each of size $n$. Let $S_{n}^{d-1}$ denote the set of $d$-permutations of size $n$. We say that $d$ is the \textit{dimension} of $\boldsymbol{\sigma}$. Moreover, the \textit{diagram} of $\sigma$ is the set of points $(i, \sigma_1(i) ,\dots, \sigma_{d-1}(i))$ for all $i \in [n]$.
\end{definition}

Note that the identity permutation is implicitly included in the diagram of a $d$-permutation, which justifies why a $d$-permutation is a $(d-1)$-tuple of permutations. For a $d$-permutation $\boldsymbol{\sigma} = (\sigma_1 ,\dots, \sigma_{d-1}),$ let $\boldsymbol{\Bar{\sigma}} = (\mathrm{Id}_n, \sigma_1, \dots, \sigma_{d-1}).$ Further, with this definition, it is natural to consider the projections of the diagram of a $d$-permutation, which is useful in defining the notion of pattern avoidance for $d$-permutations.

\begin{definition}
Given $d' \in \mathbb N$ and $\boldsymbol{i} = i_1, \dots, i_{d'} \in [d]^{d'}$, the \emph{projection on $\boldsymbol{i}$} of some $d$-permutation $\boldsymbol{\sigma}$ is the $d'$-permutation $\mathrm{proj}_{\boldsymbol{i}}(\boldsymbol{\sigma}) = (\boldsymbol{\Bar{\sigma}}_{i_2} \circ \boldsymbol{\Bar{\sigma}}_{i_1}^{-1}, \dots, \boldsymbol{\Bar{\sigma}}_{i_{d'}} \circ \boldsymbol{\Bar{\sigma}}_{i_1}^{-1} ).$
\end{definition}

We say that a projection is \emph{direct} if $i_1 < \dots < i_{d'}$ and \emph{indirect} otherwise.

\begin{remark}
There are only three direct projections of dimension $2$ of a $3$-permutation $\boldsymbol{\sigma} = (\sigma, \sigma')$. Namely, they are $\sigma$, $\sigma'$, and $\sigma' \circ \sigma^{-1}$.
\end{remark}

In the remainder of the section, we use the projection of a $3$-permutation $\boldsymbol{\sigma} = (\sigma, \sigma')$ to refer to the projection $\sigma' \circ \sigma^{-1}$. As such, we will also refer to $\sigma' \circ \sigma^{-1}$ as $\proj(\sigma, \sigma')$ for ease of notation.

In this paper, we will also consider the projection $\pi' \circ \pi^{-1}$ for subpermutations $\pi, \pi'$. We write $\rho = \pi' \circ \pi^{-1}$ to mean the composition of the standardization of $\pi$ and $\pi'$,  rewritten to be the elements of $\pi'$ with the same relative order. For example, if $\pi = 5467$ and $\pi' = 2365$, then $\pi' \circ \pi^{-1}$ is calculated as following. We first standardize $\pi$ and $\pi'$ to be $2134$ and $1243$, respectively. Then the composition of these two permutations is $2143$. Finally, we rewrite the permutation to be the elements of $\pi'$ with the same order as $2143$, giving $3265$ to be the final projection.

Using direct projections, Bonichon and Morel \cite{bonichon2022baxter} introduced the following definition of pattern avoidance, which is consistent with the existing concept of pattern avoidance for regular permutations.

\begin{definition}
Let $\boldsymbol{\sigma} = (\sigma_1 ,\dots, \sigma_{d-1}) \in S_{n}^{d-1}$ and $\boldsymbol{\pi} = (\pi_1, \dots, \pi_{d'-1}) \in S_{k}^{d'-1}$, where $k \leq n$. We say that the $d$-permutation $\boldsymbol{\sigma}$ \emph{contains} the pattern $\boldsymbol{\pi}$ if there exists a direct projection $\boldsymbol{\sigma'}$ of dimension $d'$ and indices $c_1 < \dots < c_k$ such that $\boldsymbol{\sigma'}_i(c_1)  \cdots \boldsymbol{\sigma'}_i(c_k)$ is order-isomorphic to $\pi_i$ for all $i$. We say a $d$-permutation \emph{avoids} a pattern if it does not contain it.
\end{definition}

For example, the $3$-permutation $(4231, 2413)$ avoids the pattern $123$ because neither the permutations $4231$, $2413$, nor the projection $2413 \circ 4231^{-1} = 3412$ contains an occurrence of $123$. Furthermore, note that the $3$-permutation $(1432, 3124)$ contains the pattern $231$, because despite $1432$ and $3124$ avoiding an occurrence of $231$, the projection $3124 \circ 1432^{-1} = 3421$ has an occurrence of $231$.

Given $m$ patterns $\boldsymbol{\pi_1}, \dots, \boldsymbol{\pi_m}$, we write $S_{n}^{d-1}(\boldsymbol{\pi_1}, \dots, \boldsymbol{\pi_m})$ to mean the set of $d$-permutations of size $n$ that simultaneously avoid $\boldsymbol{\pi_1}, \dots, \boldsymbol{\pi_m}$. 

Bonichon and Morel \cite{bonichon2022baxter} also noted symmetries on $d$-permutations that correspond to symmetries on the $d$-dimensional cube. In particular, these symmetries are counted by signed permutation matrices of dimension $d$. Such a signed permutation matrix is a square matrix with entries consisting of $-1, 0$, or $1$ such that each row and column contain exactly one nonzero element. We call $\textit{d-Sym}$ the set of such signed permutation matrices of size $d$.

This allows us to extend the well-known definitions of Wilf-equivalence and trivial Wilf-equivalence to higher dimensions.

\begin{definition}
We say that two sets of patterns $\boldsymbol{\pi_1}, \dots, \boldsymbol{\pi_k}$ and $\boldsymbol{\tau_1}, \dots, \boldsymbol{\tau_\ell}$ are \emph{d-Wilf-equivalent} if $|S_{n}^{d-1}(\boldsymbol{\pi_1}, \dots, \boldsymbol{\pi_k})| = |S_n^{d-1}(\boldsymbol{\tau_1}, \dots, \boldsymbol{\tau_\ell})|$. Moreover, these patterns are \emph{trivially d-Wilf-equivalent} if there exists a symmetry $s \in \textit{d-Sym}$ that maps $S_{n}^{d-1}(\boldsymbol{\pi_1}, \dots, \boldsymbol{\pi_k})$ to $S_n^{d-1}(\boldsymbol{\tau_1}, \dots, \boldsymbol{\tau_\ell})$ bijectively.
\end{definition}

In the following sections, we will only work with $3$-permutations avoiding $2$-permutations.

\section{Enumeration of Pattern Avoidance Classes of at most size 2}
\label{sec:enumeration}

Bonichon and Morel \cite{bonichon2022baxter} proposed the problem of enumerating sequences of $3$-permutations avoiding at most two patterns of size 2 or 3. They provided Table \ref{double avoidance}, conjecturing the recurrences in the last four rows and leaving the remainder as open problems.

\begin{table}[htp]
    \centering
    \begin{tabular}{|c | c | c | c | c|} 
 \hline
 Patterns & \#TWE & Sequence & OEIS Sequence & Comment \\ [0.5ex] 
 \hline\hline
 12 & 1 & $1,0,0,0,0,\dots$ & & \cite{bonichon2022baxter} \\ 
 \hline
 21 & 1 &  $1,1,1,1,1,\dots$ & & \cite{bonichon2022baxter} \\
 \hline
 123 & 1 & $1,4,20,100,410,1224,2232, \dots$ & & Not in OEIS \\
 \hline
 132 & 2 & $1,4,21,116,646,3596,19981, \dots$ & & Not in OEIS \\
 \hline
 231 & 2 & $1,4,21,123,767,4994,33584, \dots$ & & Not in OEIS \\
 \hline
 321 & 1 & $1,4,21,128,850,5956,43235, \dots$ & & Not in OEIS \\
 \hline
 $123, 132$ & 2 & $1,4,8,8,0,0,0, \dots$  & & Terminates after $n=4$ \\
 \hline
 $123, 231$ & 2 & $1,4,9,6,0,0,0, \dots$ & & Terminates after $n=4$ \\
 \hline
 $123, 321$ & 1 & $1,4,8,0,0,0,0, \dots$ & & Terminates after $n=3$ \\
 \hline
 $132, 213$ & 1 & $1,4,12,28,58,114,220, \dots$ & \href{http://oeis.org/A356728}{A356728} & Theorem \ref{132,213} \\
 \hline
 $132, 231$ & 4 & $1,4,12,32,80,192,448, \dots$ & \href{http://oeis.org/A001787}{A001787} & Theorem \ref{132,231} \\
 \hline
 $132, 321$ & 2 & $1,4,12,27,51,86,134, \dots$ & \href{http://oeis.org/A047732}{A047732} & Theorem \ref{132,321} \\
 \hline
 $231, 312$ & 1 & $1,4,10,28,76,208,568, \dots$ & \href{http://oeis.org/A026150}{A026150} & Theorem \ref{231,312} \\
 \hline
 $231, 321$ & 2 & $1,4,12,36,108,324,972, \dots$ & \href{http://oeis.org/A003946}{A003946} & Theorem \ref{231,321} \\ 
 \hline
\end{tabular}
    \caption{Sequences of $3$-permutations avoiding at most two patterns of size 2 or 3. The second column indicates the number of trivially Wilf-equivalent classes.}
    \label{double avoidance}
\end{table}

In all of the following theorems, we take constructive approaches to prove recurrence relations. Given an element $\boldsymbol{\sigma}$ in $S_n^2(\pi_1, \pi_2)$, we attempt to construct elements in $S_{n+1}^2(\pi_1, \pi_2)$ via inserting the maximal element $n+1$ into the permutations in $\boldsymbol{\sigma}$. Note that if a permutation $\sigma \in S_n$ contains a pattern $\pi$, then adding the maximal element $n+1$ anywhere into $\sigma$ still contains $\pi$. Similarly, if a permutation $\sigma \in S_n$ avoids a pattern $\pi$, then removing the maximal element $n$ from $\sigma$ will still avoid~$\pi$. 

However, it should be noted that it is possible to have a $3$-permutation $(\sigma, \sigma')$ that does not avoid a set of permutations $\{ \pi_1, \dots, \pi_m \}$ and inserting the maximum element $(n+1)$ into both $\sigma$ and $\sigma'$ results in a $3$-permutation that avoids these patterns. For example, the $3$-permutation $(312,123)$ contains $231$, but $(3124, 4123)$ avoids both $231$ and $321$. Although in the following proofs we aim to construct elements in $S_{n+1}^2(\pi_1, \pi_2)$ from $S_n^2(\pi_1, \pi_2)$, we will prove that for each set of patterns $\{\pi_1, \pi_2\}$, it is impossible to insert $n+1$ into $\sigma$ and $\sigma'$ of a $3$-permutation $(\sigma, \sigma')$ containing $\pi_1$ or $\pi_2$ such that the resulting $3$-permutation avoids $\pi_1$ and $\pi_2$. It is clear that if $\sigma$ or $\sigma'$ contains these patterns, then inserting $n+1$ anywhere into these permutations will still contain these patterns. Hence, it is enough to show that for a $3$-permutation $(\sigma, \sigma')$ where $\sigma$ and $\sigma'$ avoid $\pi_1$ and $\pi_2$ but $\sigma' \circ \sigma^{-1}$ contains either pattern, inserting $n+1$ anywhere will result in a $3$-permutation which still contains either pattern. In the following proofs, note that given a $3$-permutation $(\sigma, \sigma')$, if the maximal element $n+1$ is inserted into the same position in both $\sigma$ and $\sigma'$, then $n+1$ is inserted at the end of the projection $\sigma' \circ \sigma^{-1}$.

\begin{theorem}\label{132,231}
Let $a_n = |S_n^2(132,231)|$. Then $a_n$ satisfies the recurrence relation $a_{n+1} = 2a_n + 2^{n}$ with initial term $a_1 =1$, which corresponds with OEIS sequence \href{http://oeis.org/A001787}{A001787}.
\end{theorem}
\begin{proof}
Given any $\boldsymbol{\sigma} = (\sigma, \sigma') \in S^2_n(132,231)$, we construct an element of $S^2_{n+1}(132,231)$ by inserting the maximal element $n+1$ in both $\sigma$ and $\sigma'$. To avoid both $132$ and $231$, the maximal element $n+1$ must be inserted into either the beginning or end of $\sigma$ and $\sigma'$; otherwise if there are elements on both sides of $n+1$, then there must be either an occurrence of $132$ or $231$.

Appending the maximal element $n+1$ onto the left of both $\sigma$ and $\sigma'$ or onto the right of both $\sigma$ and $\sigma'$ also avoids $132$ and $231$. In other words, $(\sigma (n+1), \sigma'(n+1))$ and $((n+1)\sigma, (n+1)\sigma')$ both still avoid $132$ and $231$. This contributes $2a_{n}$ different $3$-permutations in $S_{n+1}^2(132,231)$.

We further make the following claim:

\begin{claim}\label{132 231 claim 1}
The $3$-permutation $(\sigma(n+1), (n+1)\sigma')$ avoids $132$ and $231$ if and only if $\sigma$ is $\mathrm{Id}_n$ and $\sigma' \in S_n^1(132,231)$.
\end{claim} 
\begin{proof}
For the forwards direction, suppose that $(\sigma(n+1), (n+1)\sigma')$ avoids $132$ and $231$. Now writing the projection $\proj(\sigma(n+1), (n+1)\sigma') = (\sigma_L (n+1) \sigma_R)$ for some subpermutations $\sigma_L$ and $\sigma_R$, note that $\sigma_R$ is nonempty, and using the reasoning mentioned above, $\sigma_L$ is empty. Otherwise, $(\sigma_L (n+1) \sigma_R)$ contains an occurrence of either $132$ or $231$. Thus, $\sigma$ begins with the minimal element $1$. But since $\sigma$ is forced to avoid the $132$ pattern, $\sigma$ is forced to be consecutive and becomes the identity permutation.

For the backwards direction, both $\Id_{n+1}$ and $((n+1)\sigma')$ still avoid $132$ and $231$. Further, the projection $\proj(\Id_{n+1}, (n+1)\sigma')$ evaluates to $(n+1)\sigma'$, which also still avoids $132$ and $231$. 
\end{proof}

Similarly, we have by symmetry the following claim:

\begin{claim}\label{132 231 claim 2}
The $3$-permutation $((n+1)\sigma, \sigma'(n+1))$ avoids $132$ and $231$ if and only if $\sigma$ is $\mathrm{rev}(\mathrm{Id}_n)$ and $\sigma' \in S_n^1(132,231)$.
\end{claim} 

Now we show that for $(\sigma, \sigma') \in S_n^2 \setminus S_n^2(132,231)$, we cannot obtain an element in $ S_{n+1}^2(132,231)$ by inserting the maximal element $n+1$ anywhere in $\sigma$ and $\sigma'$. We will assume that $\sigma$ and $\sigma'$ avoid these patterns but $\sigma' \circ \sigma^{-1}$ does not. Recall that we are forced to insert $n+1$ onto the left or right of $\sigma$ and $\sigma'$. Inserting $n+1$ onto the left of both $\sigma$ and $\sigma'$ or onto the right of both $\sigma$ and $\sigma'$ gives a $3$-permutation with a projection containing $\sigma' \circ \sigma^{-1}$, which contains either $132$ or $231$. Now for the $3$-permutation $((n+1)\sigma, \sigma'(n+1))$, our reasoning in Claim \ref{132 231 claim 2} gives that either the projection contains $132$ or $231$ or that $\sigma = \rev(\Id_n)$. In the latter case, $\proj(\sigma, \sigma')$ would become $\rev(\sigma')$, which avoids $132$ and $231$, a contradiction. Claim \ref{132 231 claim 1} provides a similar reasoning on how $(\sigma (n+1), (n+1)\sigma')$ contains $132$ or $231$. Because inserting $n+1$ anywhere else in $\sigma$ and $\sigma'$ gives an occurrence of $231$ or $132$, we ensure that elements of $S_n^2$ not belonging in $S_n^2(132,231)$ cannot result in elements belonging in $S_{n+1}^2(132,231)$ when we insert the maximal element $n+1$ anywhere into $\sigma$ and $\sigma'$.

Thus, we have shown that given any $3$-permutation $\boldsymbol{\sigma} = (\sigma, \sigma') \in S^2_n(132,231)$, we can construct two elements in $S^2_{n+1}(132,231)$; furthermore, we can construct one additional element in $S^2_{n+1}(132,231)$ if and only if $\sigma' \in S^1_n(132,231)$ and $\sigma$ is $\Id_n$ or $\rev(\Id_n)$. Simion and Schmidt \cite{simion1985restricted} have shown that $|S^1_n(132,231)| = 2^{n-1}$. In the cases where $\sigma$ is $\Id_n$ or $\rev(\Id_n)$, it follows that $\boldsymbol{\sigma}$ avoids $132$ and $231$ if and only if $\sigma'$ avoids these patterns, and hence, it follows that \begin{align*}
    a_{n+1} = 2a_n + 2^n. & \qedhere
\end{align*}

\end{proof}

\begin{theorem}\label{132,321}
Let $a_n = |S^2_n(132,321)|$. Then $a_n$ satisfies the recurrence $a_{n+1} = a_n +n(n+2)$ with initial term $a_1 = 1$, which corresponds with the OEIS sequence \href{http://oeis.org/A047732}{A047732}.
\end{theorem}
\begin{proof}
Let us write $\boldsymbol{\sigma} = (\sigma, \sigma') \in S^2_n(132,321)$ of the form $(\sigma_L n \sigma_R, \sigma_L' n \sigma_R')$. We construct an element of $S^2_{n+1}(132,321)$ by inserting the maximal element $n+1$ in both $\sigma$ and $\sigma'$.

Inserting $n+1$ onto the end of $\sigma$ and $\sigma'$ always constructs a $132$-avoiding and $321$-avoiding $3$-permutation, and this contributes $a_n$ different $3$-permutations to $S^2_{n+1}(132,321)$. 

We also have the following three cases. Figure \ref{fig:132,321 block} gives an example on what $\sigma$ may look like.

\begin{enumerate}
    \item $\sigma_R$ and $\sigma_R'$ are both nonempty and $\sigma_L n$, $\sigma_L' n$, $\sigma_R$, and $\sigma_R'$ are all consecutively increasing. 
    
    \item Exactly one of $\sigma_R$, $\sigma_R'$ is empty and the other is of the form $\sigma_L n \sigma_R$, where $\sigma_L n$ and $\sigma_R$ are consecutively increasing.
    
    \item Both $\sigma_R$, $\sigma_R'$ are empty.
\end{enumerate}

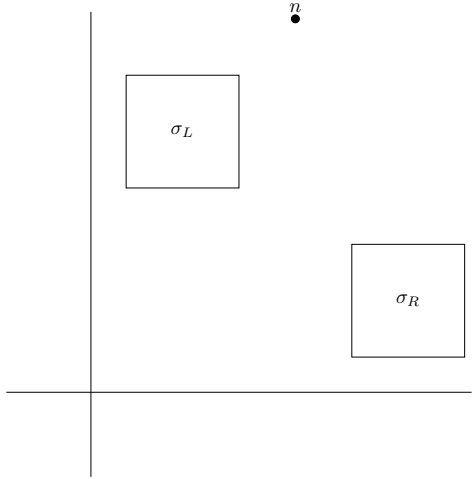
\begin{figure}[h]
    \centering
    \scalebox{0.75}{
    \begin{tikzpicture}
        \foreach [evaluate=\i as \x using int(\i-1)]\i in {0,1,...,8}
        {
            \foreach [evaluate=\j as \y using int(\j-1)] \j in {0,1,...,8}
            {
                \node at (\i,\j)[name=perm-\x-\y,]{};
            }
        }

        \draw ([xshift=-5mm]perm-1--1.south west)--([xshift=-5mm]perm-1-7.north west);
        \draw ([yshift=5mm]perm--1-0.south west)--([yshift=5mm]perm-7-0.south east);

        \draw[draw=black] (2,5) rectangle ++(2,2);
        \node at (3,6) {$\sigma_L$};

        \filldraw [black] (5,8) circle (2pt) node[anchor=south]{$n$};

        \draw[draw=black] (6,2) rectangle ++(2,2);
        \node at (7,3) {$\sigma_R$};

\end{tikzpicture}
}
    \caption{An example of what $\sigma_L n \sigma_R$ may look like when avoiding $132$ and $321$.}
    \label{fig:132,321 block}
\end{figure}

First we show that when $\boldsymbol{\sigma}$ does not belong to any of these cases, inserting the maximal element $n+1$ into  $\boldsymbol{\sigma}$ cannot avoid these patterns. Let $\boldsymbol{\sigma} = (\sigma_L n \sigma_R, \sigma_L' n \sigma_R')$.  If $\sigma_L n$ is increasing but not consecutively increasing, $\sigma$ must contain an occurrence of $132$. Further, note that $\sigma_R$ must not contain a $21$ pattern, and in the case where $\sigma_L n$ is consecutively increasing, $\sigma_R$ is consecutively increasing.
 
 So the only case we need to consider is where there is an occurrence of $ab$ in $\sigma_L$ for some $b <a$; the argument for where $\sigma_L'$ contains the pattern $21$ is similar. 
 Note that every element of $\sigma_L$ and $\sigma_L'$ still must be greater than every element of $\sigma_R$ and $\sigma_R'$, respectively; otherwise, they would contain an occurrence of $132$. This implies that $\sigma_R$ cannot contain elements in the interval $(b, n)$. Similarly, if $\sigma_R$ contains elements in the interval $[1, b)$, then $\boldsymbol{\sigma}$ would contain an occurrence of $321$. Thus, $\sigma_R$ is empty. Inserting $n+1$ to the left of $a$ gives an occurrence of $321$. And inserting $n+1$ to the right of $a$ gives an occurrence of $132$. Hence, nothing outside these cases avoids $132$ and $321$.

Now we present each case:
\begin{enumerate}
    \item We claim that only $(\sigma_L n (n+1) \sigma_R, \sigma_L' n (n+1) \sigma_R')$ avoids $132$ and $321$.

    Let us write $\gamma$ as the projection $(\sigma_L' n (n+1) \sigma_R') \circ (\sigma_L n (n+1) \sigma_R)^{-1}$. Since $\sigma_L$ and $\sigma_R$ are consecutive and $\sigma_R$ must start with $1$, note that in $\gamma$, either $n+1$ is right-adjacent to $n$, or $\gamma$ begins with $n+1$. In the former case, $\gamma$ is of the form $\pi_L \pi_R$, where $\pi_L$ and $\pi_R$ are consecutively increasing and each element of $\pi_L$ is greater than each element of $\pi_R$. Note that $\pi_R$ may be empty, in which case $\pi_L$ is the identity permutation. It is clear that $\pi_L \pi_R$ avoids $132$ and $321$. In the latter case, $\gamma$ is of the form $(n+1) \sigma_R' \sigma_L' n$. But $\sigma_R' \sigma_L' n$ is strictly increasing, and hence, $\gamma$ also avoids $132$ and $321$. Therefore, the $3$-permutation $(\sigma_L n (n+1) \sigma_R, \sigma_L' n (n+1) \sigma_R')$ avoids these patterns too.
    
    Now we show that inserting $n+1$ into $\boldsymbol{\sigma}$ anywhere else cannot result in an element in $S_{n+1}^2(132,321)$. In particular, we show that we are forced to insert $n+1$ at the end of $\sigma$ and $\sigma'$ or right-adjacent to $n$ in these two permutations. Otherwise, since $\sigma_L$, $\sigma_L'$, $\sigma_R$, and $\sigma_R'$ are all consecutively increasing, $\sigma$ or $\sigma'$ would contain $132$. If we insert $n+1$ to the left of $\sigma$ or $\sigma'$, we would have an occurrence of $321$.
    
    Now it is sufficient to show $(\sigma_L n \sigma_R (n+1), \sigma_L' n (n+1) \sigma_R')$ and $(\sigma_L n (n+1) \sigma_R, \sigma_L' n \sigma_R' (n+1))$ do not avoid $132$ and $321$.
    
    To see the former, we take the projection $\rho$, and depending on the lengths of $\sigma_L$ and $\sigma_L'$, we have the following three cases. The projections in each case are illustrated in Figure \ref{fig:132,321 p2}.

    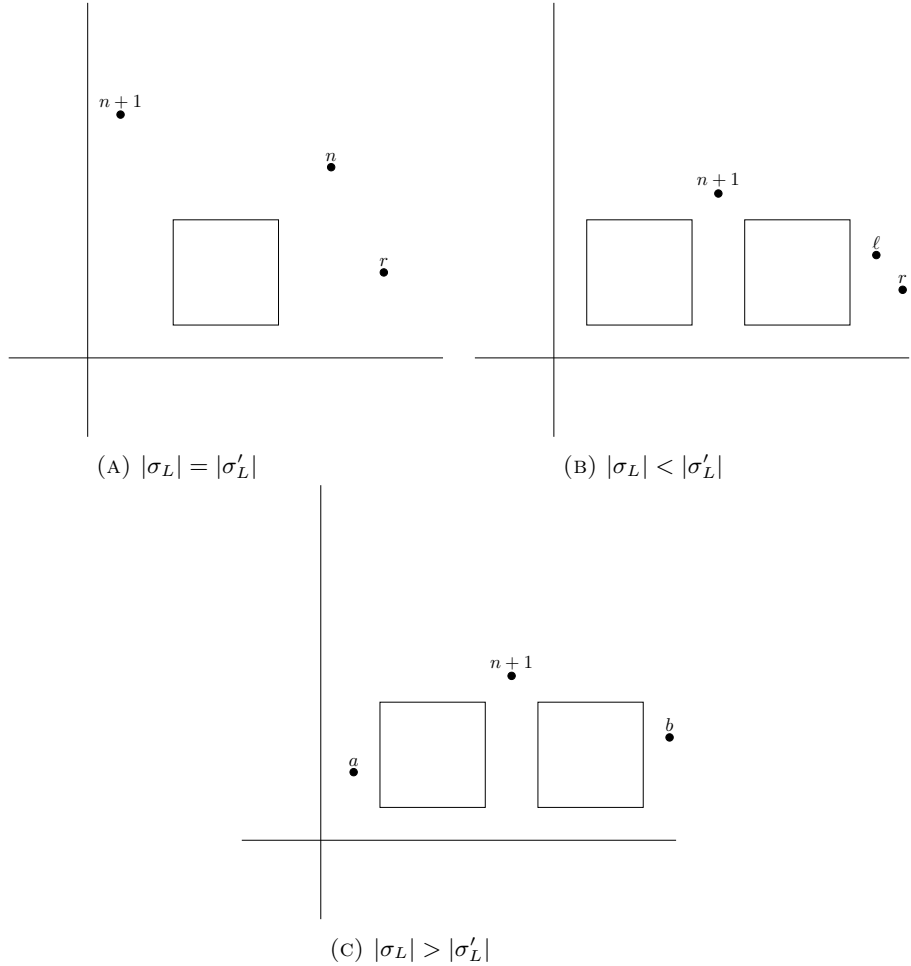
\begin{figure}[h]
        \centering
        \begin{subfigure}{0.3\textwidth}
         \scalebox{0.7}{
        \begin{tikzpicture}
        \foreach [evaluate=\i as \x using int(\i-1)]\i in {0,1,...,8}
        {
            \foreach [evaluate=\j as \y using int(\j-1)] \j in {0,1,...,8}
            {
                \node at (\i,\j)[name=perm-\x-\y,]{};
            }
        }

        \draw ([xshift=-5mm]perm-1--1.south west)--([xshift=-5mm]perm-1-7.north west);
        \draw ([yshift=5mm]perm--1-0.south west)--([yshift=5mm]perm-7-0.south east);

        \filldraw [black] (2,6) circle (2pt) node[anchor=south]{$n+1$};

        \draw[draw=black] (3,2) rectangle ++(2,2);
        \node at (4,3){};

        \filldraw [black] (6,5) circle (2pt) node[anchor=south]{$n$};

        \filldraw [black] (7,3) circle (2pt) node[anchor=south]{$r$};

\end{tikzpicture}
}
        \caption{$|\sigma_L| = |\sigma_L'|$}
        \end{subfigure}\hspace{0.1\textwidth}%
\begin{subfigure}{0.3\textwidth}
        \centering
         \scalebox{0.7}{
        \begin{tikzpicture}
        \foreach [evaluate=\i as \x using int(\i-1)]\i in {0,1,...,8}
        {
            \foreach [evaluate=\j as \y using int(\j-1)] \j in {0,1,...,8}
            {
                \node at (\i,\j)[name=perm-\x-\y,]{};
            }
        }

        \draw ([xshift=-5mm]perm-1--1.south west)--([xshift=-5mm]perm-1-7.north west);
        \draw ([yshift=5mm]perm--1-0.south west)--([yshift=5mm]perm-7-0.south east);

        \draw[draw=black] (2,2) rectangle ++(2,2);
        
        \filldraw [black] (4.5,4.5) circle (2pt) node[anchor=south]{$n+1$};

        \draw[draw=black] (5,2) rectangle ++(2,2);

        \filldraw [black] (7.5,3.33) circle (2pt) node[anchor=south]{$\ell$};

        \filldraw [black] (8,2.67) circle (2pt) node[anchor=south]{$r$};

\end{tikzpicture}
}
        \caption{$|\sigma_L| < |\sigma_L'|$}
        \end{subfigure}\hspace{0.1\textwidth}%

        \begin{subfigure}{0.3\textwidth}
        \centering
         \scalebox{0.7}{
        \begin{tikzpicture}
        \foreach [evaluate=\i as \x using int(\i-1)]\i in {0,1,...,8}
        {
            \foreach [evaluate=\j as \y using int(\j-1)] \j in {0,1,...,8}
            {
                \node at (\i,\j)[name=perm-\x-\y,]{};
            }
        }

        \draw ([xshift=-5mm]perm-1--1.south west)--([xshift=-5mm]perm-1-7.north west);
        \draw ([yshift=5mm]perm--1-0.south west)--([yshift=5mm]perm-7-0.south east);

        \filldraw [black] (2,2.67) circle (2pt) node[anchor=south]{$a$};
        
        \draw[draw=black] (2.5,2) rectangle ++(2,2);
        
        \filldraw [black] (5,4.5) circle (2pt) node[anchor=south]{$n+1$};

        \draw[draw=black] (5.5,2) rectangle ++(2,2);

        \filldraw [black] (8,3.33) circle (2pt) node[anchor=south]{$b$};

\end{tikzpicture}
}
        \caption{$|\sigma_L| > |\sigma_L'|$}
        \end{subfigure}

        \caption{Illustrations of the projection of $(\sigma_L n \sigma_R (n+1), \sigma_L' n (n+1) \sigma_R')$ for the cases $|\sigma_L| = |\sigma_L'|$, $|\sigma_L| < |\sigma_L'|$, and $|\sigma_L| > |\sigma_L'|$, respectively.}
        \label{fig:132,321 p2}
    \end{figure}

        \begin{enumerate}
            \item $|\sigma_L| = |\sigma_L'|$. Then the last two elements of $\rho$ must be $n r$, where $r$ is an element in $\sigma_R'$. Note that the maximum element $n+1$ must appear before this occurrence, and hence, $\rho$ contains an occurrence of $321$.
            
            \item $|\sigma_L| < |\sigma_L'|$. Then the last two elements of $\rho$ must be $\ell r$, where $\ell$ is an element in  $\sigma_L'$ and $r$ is an element in $\sigma_R'$. Note that the maximum element $n+1$ must appear somewhere in $\rho$ before this occurrence, and thus, $\rho$ contains an occurrence of $321$.
            
            \item $|\sigma_L| > |\sigma_L'|$. Then $\rho$ begins with an element $a$ in $\sigma_R'$ and ends with a larger element $b$ in $\sigma_R'$. However, the maximum element $n+1$ must appear in between these elements, and we conclude that $\rho$ contains an occurrence of $132$.
        \end{enumerate}

        A similar reasoning shows that $(\sigma_L n (n+1) \sigma_R, \sigma_L' n \sigma_R' (n+1))$ does not avoid $132$ and $321$.
    
    There are $(n-1)$ ways to choose $\sigma_L n \sigma_R$ and $\sigma_L' n \sigma_R'$. This case contributes $(n-1)^2$ distinct $3$-permutations to $S^2_{n+1}(132,321)$.
    
    \item Let $\sigma_R$ be empty - the case where $\sigma_R'$ is empty follows a similar reasoning. Then we claim that only the $3$-permutations $((n+1)\sigma_L n, \sigma_L' n (n+1) \sigma_R')$ and $(\sigma_L n(n+1), \sigma_L' n (n+1) \sigma_R')$ avoid $132$ and $321$. 
    
    Checking that both of these $3$-permutations avoid $132$ and $321$ uses a similar argument to the previous case. Now we show that inserting $n+1$ anywhere else in $\boldsymbol{\sigma}$ cannot avoid the patterns $132$ and $321$. In particular, we must insert $n+1$ into the beginning or end of $\sigma$ and either right-adjacent to $n$ or at the end in $\sigma'$. Hence, it is sufficient to show that $((n+1)\sigma_L n, \sigma_L' n \sigma_R' (n+1))$ does not avoid $132$ and $321$. 
    
    Now taking $\proj((n+1)\sigma_L n, \sigma_L' n \sigma_R' (n+1))$ gives us a permutation of the form $\pi (n+1) c$, where $c$ is the first element of $\sigma_L' n$ and $\pi$ is some subpermutation. Since $\sigma_R'$ is nonempty, $\pi$ contains elements in $\sigma_R'$, and this composition contains an instance of $132$. 
    
    Since there are $n-1$ ways to choose $\sigma_L' n \sigma_R'$, this contributes $2(n-1)$ many $3$-permutations to $S^2_{n+1}(132,321)$. A similar argument holds for when $\sigma_R'$ is empty and $\sigma_R$ is nonempty. Hence, this case contributes $4(n-1)$ many $3$-permutations in total to $S^2_{n+1}(132,321)$.
    
    \item Note that we must insert $n+1$ into the beginning or end of $\sigma$ and $\sigma'$. However, the only $3$-permutations that avoid $132$ and $321$ obtained by inserting $n+1$ into both $\sigma$ and $\sigma'$ are the following: $(\Id_n (n+1), (n+1) \Id_n)$, $((n+1)\Id_n, (n+1) \Id_n)$, and $((n+1)\Id_n, \Id_n (n+1))$. To see this, let us consider $((n+1) \sigma, \sigma' (n+1))$. Inserting $n+1$ at the beginning of $\sigma$ forces $\sigma$ to be the identity. Using a similar reasoning presented in Case 2, unless $\sigma' = \Id_n$, inserting $n+1$ at the end of $\sigma'$ will cause $\proj((n+1) \sigma, \sigma' (n+1))$ to contain $132$. A similar argument can be used for the other $3$-permutations, and checking that these avoid $132$ and $321$ is simple. Hence, this case contributes $3$ new elements in $S^2_{n+1}(132,321)$.
\end{enumerate}

Now we show that for $(\sigma, \sigma') \in S_n^2 \setminus S_n^2(132,321)$, we cannot obtain an element in $ S_{n+1}^2(132,321)$ by inserting the maximal element $n+1$ anywhere in $\sigma$ and $\sigma'$. Suppose $\sigma$ and $\sigma'$ avoid $132$ and $321$ but $\sigma' \circ \sigma^{-1}$ contains either pattern. We iterate the same cases as above, noting that if none of these cases hold, inserting $n+1$ anywhere will give a $3$-permutation containing $132$ or $321$. 

\begin{enumerate}
    \item We show that it is impossible to have $\sigma$ and $\sigma'$ avoid these patterns but have $\sigma' \circ \sigma^{-1}$ contain them. Let  $(\sigma, \sigma') = (\sigma_L n \sigma_R, \sigma_L' n \sigma_R')$ and $\rho = \proj(\sigma, \sigma')$. Recall that $\sigma_L n$, $\sigma_L' n$, $\sigma_R$, and $\sigma_R$' are all consecutively increasing. We have the same subcases as above:
        \begin{enumerate}
            \item $|\sigma_L| = |\sigma_L'|$. Then $\rho$ is $\sigma_R' \sigma_L' n$, which avoids $132$ and $321$.

            \item $|\sigma_L| < |\sigma_L'|$. Then $\rho$ is in the form $\pi_L \pi_R$, where $\pi_L$ and $\pi_R$ are consecutively increasing and every element of $\pi_L$ is greater than every element of $\pi_R$. Note that $\rho$ avoids both $132$ and $321$.

            \item $|\sigma_L| > |\sigma_L'|$. Then $\rho$ is in the same form as the previous case.
        \end{enumerate}

        Hence, it is impossible for $\sigma$ and $\sigma'$ to avoid $132$ and $321$ while $\sigma' \circ \sigma^{-1}$ contains them.
            
    \item Let $\sigma_R$ be empty. The case where $\sigma_R'$ is empty follows a similar reasoning. The proof of Case 2 above shows that the $3$-permutation $((n+1) \sigma_L n, \sigma_L' n \sigma_R' (n+1))$ contains an instance of $132$, and it is clear that $(\sigma_L n (n+1), \sigma_L' n \sigma_R' (n+1))$ contains an occurrence of $\sigma' \circ \sigma^{-1}$. Now we consider $((n+1) \sigma_L n, \sigma_L' n (n+1) \sigma_R')$. Let $\ell$ be the first element of $\sigma_L'$. 
    Now we note that $\proj((n+1) \sigma_L n, \sigma_L' n (n+1) \sigma_R')$ is obtained from $\proj(\sigma_L n, \sigma_L' n \sigma_R')$ as follows: each element belonging to the subpermutation $\sigma_L' n$ is increased by $1$, every element in $\sigma_R'$ remains unchanged, and $\ell$ is appended onto the end of the permutation. It is clear that if $\proj(\sigma_L n, \sigma_L' n \sigma_R')$ contains an instance of $132$ or $321$, then $\proj((n+1) \sigma_L n, \sigma_L' n (n+1) \sigma_R')$ must also contain an occurrence of these patterns.

    Now we consider $(\sigma_L n (n+1), \sigma_L' n (n+1) \sigma_R')$. Let $r$ be the last element in $\sigma_R'$. Similarly, $\proj(\sigma_L n (n+1), \sigma_L' n (n+1) \sigma_R')$ is obtained from the $\proj(\sigma_L n, \sigma_L' n \sigma_R')$ as follows: each element in $\sigma_R'$ decreases by $1$, the minimal element $1$ is replaced with $n+1$, elements in $\sigma_L' n$ remain unchanged, and $r$ is appended onto the end of the permutation. If the minimal element $1$ is not a part of the occurrence of $132$ or $321$ in $\proj(\sigma_L n, \sigma_L' n \sigma_R')$, then $\proj(\sigma_L n (n+1), \sigma_L' n (n+1) \sigma_R')$ still contains an occurrence of either pattern. If the minimal element $1$ is part of a $132$ pattern, the transformation described above turns the pattern into a $321$ pattern. Now suppose the minimal element $1$ is part of a $321$ pattern. Call this sequence $xy1$. If there is an element $a$ after $xy1$, then if $y>a$, then the transformation turns $xy1$ into a $321$ pattern. If $y<a$, then we instead have a $132$ pattern. Now if $xy1$ are the last elements in the permutation, then the transformation turns this into a permutation ending in $1$, which contains an occurrence of $321$. Hence, $(\sigma_L n (n+1), \sigma_L' n (n+1) \sigma_R')$ must contain $132$ or $321$.

    \item As mentioned in Case 3 above, in order for us to insert $n+1$ anywhere else into $\sigma$ and $\sigma'$ to avoid $132$ and $321$ in the resulting $3$-permutation, we must have $\sigma = \sigma' = \Id_n$. Then $\sigma' \circ \sigma^{-1} = \Id_n$, which does not contain $132$ or $321$, a contradiction.
\end{enumerate}
    
    We conclude that \begin{align*}
        a_{n+1} &= a_n + (n-1)^2 + 4(n-1) +3 \\
        &= a_n + n(n+2). & \qedhere
    \end{align*}

\end{proof}

\begin{theorem}\label{231,312}
Let $a_n = |S^2_n(231,312)|$. Then $a_n$ satisfies the recurrence  $a_{n+1} = 2a_n + 2a_{n-1}$ with initial terms $a_1 = 1$ and $a_2 =4$, which corresponds to the OEIS sequence \href{http://oeis.org/A026150}{A026150}.
\end{theorem}
\begin{proof}

Let $\boldsymbol{\sigma} = (\sigma, \sigma') \in S_n^2(231,312)$ and write $\boldsymbol{\sigma}$ of the form $(\sigma_L n \sigma_R, \sigma_L' n \sigma_R')$. Note that each element of $\sigma_L$ and $\sigma_L'$ are less than each element of $\sigma_R$ and $\sigma_R'$, respectively. Further, $\sigma_R$ and $\sigma_R'$ have to be consecutively decreasing. If $\sigma_R$ is nonempty, $n-1$ must be right-adjacent to $n$ in $\sigma$ to avoid instances of $231$ and $312$. We then have the following cases, where $\sigma_R$ and $\sigma_R'$ may be empty. Figure \ref{fig:231,312 block} illustrates an example of what $\sigma$ may look like.

\begin{enumerate}
    \item $\boldsymbol{\sigma}$ is of the form $(\sigma_L n, \sigma_L' n)$.
        
    \item $\boldsymbol{\sigma}$ is of the form $(\sigma_L (n-1) n, \sigma_L' n (n-1)  \sigma_R')$.
    
    \item $\boldsymbol{\sigma}$ is of the form $(\sigma_L n (n-1) \sigma_R, \sigma_L' (n-1) n)$.
    
    \item $\boldsymbol{\sigma}$ is of the form $(\sigma_L n (n-1) \sigma_R, \sigma_L' n (n-1) \sigma_R')$.
\end{enumerate}

\begin{figure}[h]
    \centering
    \scalebox{0.75}{
    \begin{tikzpicture}
        \foreach [evaluate=\i as \x using int(\i-1)]\i in {0,1,...,8}
        {
            \foreach [evaluate=\j as \y using int(\j-1)] \j in {0,1,...,8}
            {
                \node at (\i,\j)[name=perm-\x-\y,]{};
            }
        }

        \draw ([xshift=-5mm]perm-1--1.south west)--([xshift=-5mm]perm-1-7.north west);
        \draw ([yshift=5mm]perm--1-0.south west)--([yshift=5mm]perm-7-0.south east);

        \draw[draw=black] (2,2) rectangle ++(2,2);
        \node at (3,3) {$\sigma_L$};

        \filldraw [black] (4.5,7.5) circle (2pt) node[anchor=south]{$n$};

        \filldraw [black] (5,7) circle (2pt) node[anchor=west]{$n-1$};

        \draw[draw=black] (5.5,4.5) rectangle ++(2,2);
        \node at (6.5,5.5) {$\sigma_R$};

\end{tikzpicture}
}
    \caption{An example of what $\sigma_L n (n-1) \sigma_R$ may look like when avoiding $231$ and~$312$.}
    \label{fig:231,312 block}
\end{figure}
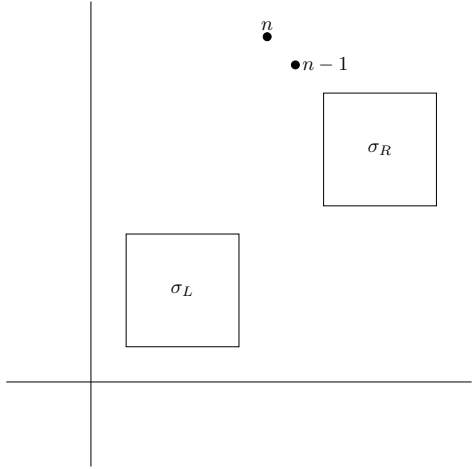

Now we present each case:

\begin{enumerate}
    \item $(\sigma, \sigma') = (\sigma_L n, \sigma_L' n)$.
    
    The maximal element $n+1$ must be inserted adjacent to $n$ in both $\sigma$ and $\sigma'$. If not, then there would be an occurrence of $312$. By evaluating their projections, we can verify that the following $3$-permutations avoid $231$ and $312$: $(\sigma_L n (n+1), \sigma_L' n (n+1))$, ${(\sigma_L n (n+1), \sigma_L' (n+1)n)}$, ${(\sigma_L (n+1) n, \sigma_L' n (n+1))},$ and $(\sigma_L (n+1)n, \sigma_L' (n+1)n)$. Thus, each instance of $\boldsymbol{\sigma}$ in this case contributes $4$ new $3$-permutations that avoid $231$ and $312$.
        
        \item $(\sigma, \sigma')= (\sigma_L (n-1) n, \sigma_L' n (n-1)  \sigma_R')$.
        
        Then $n (n-1) \sigma_R'$ must be consecutively decreasing. Note that appending the maximal element ${n+1}$ onto the end of $\sigma$ and $\sigma'$ also avoids $231$ and $312$. In other words, the $3$-permutation $(\sigma_L (n-1) n (n+1), \sigma_L' n (n-1)  \sigma_R' (n+1))$ avoids $231$ and $312$. In addition, the $3$-permutation $(\sigma_L (n-1) n (n+1), \sigma_L' (n+1) n (n-1)  \sigma_R')$ also avoids $231$ and $312$. 
        
        To see this, we first evaluate the projection of $(\sigma_L (n-1) n, \sigma_L' n (n-1)  \sigma_R')$. Call this $\gamma$. As shown in Figure \ref{fig1}, we can subdivide $\sigma_L$ into $\pi_L$ and $\pi_R$, where $|\pi_L| = |\sigma_L'|$. Recall that $n (n-1) \sigma_R'$ is consecutively decreasing, and thus, the sequence $\pi_R (n-1) n$ must be consecutively increasing to prevent $\gamma$ from containing an instance of $231$ or $312$. 
        Note that this implies that each element of $\pi_L$ is less than each element of $\pi_R$. Thus, $\gamma$ is of the form $(\sigma_L' \circ \pi_L^{-1}) n (n-1) \sigma_R'$, which we note must avoid $231$ and $312$ since $(\sigma, \sigma')$ avoids these patterns. Further, note that $(\sigma_L (n-1) n (n+1), \sigma_L' (n+1) n (n-1)  \sigma_R')$ avoids $231$ and $312$ because its projection is of the form $(\sigma_L' \circ \pi_L^{-1})(n+1)n (n-1) \sigma_R'$, which also avoids these patterns.

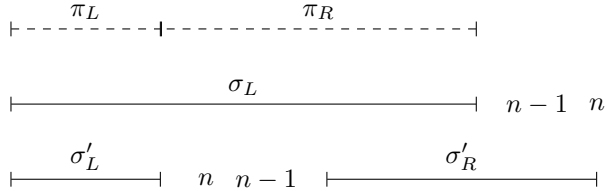
\begin{figure}[htp]
\centering
    \begin{tikzpicture}[coo/.style={coordinate}]
	\path (0,0) coordinate (x1) --++ (2,0) coordinate (x2)
			--++ (0.6,0) coordinate (x3)
			--++ (0.8,0) coordinate (x4)
			--++ (0.8,0) coordinate (x5)
			--++ (2,0) coordinate (x6)
			--++ (0.8,0) coordinate (x7)
			--++ (0.8,0) coordinate (x8);
	\foreach \i in {1,2,3}
		\coordinate[] (L\i) at (0,-\i);

	\draw[|-|] (L2-|x1) -- (L2-|x6) node[above,pos=0.5] {$\sigma_L$};
	\node at (L2-|x7) {$n-1$};
	\node at (L2-|x8) {$n$};
		 
	\draw[|-|] (L3-|x1) -- (L3-|x2) node[above,pos=0.5] {$\sigma_L'$};
	\node at (L3-|x3) {$n$};
	\node at (L3-|x4) {$n-1$};
	\draw[|-|] (L3-|x5) -- (L3-|x8) node[above,pos=0.5] {$\sigma_R'$}; 
	
	\draw[|-|,dashed] (L1-|x1) -- (L1-|x2) node[above,pos=0.5] {$\pi_L$};
	\draw[|-|,dashed] (L1-|x6) -- (L1-|x2) node[above,pos=0.5] {$\pi_R$};

\end{tikzpicture}
\caption{The two-line notation used to evaluate $(\sigma_L' n (n-1) \sigma_R') \circ (\sigma_L (n-1)n)^{-1}$. The second line represents $\sigma$ and the last line represents $\sigma'$.}
\label{fig1}
\end{figure}
        
        Now we show that inserting $n+1$ anywhere else in $\boldsymbol{\sigma}$ cannot produce a $3$-permutation that avoids $231$ and $312$. We must insert $n+1$ adjacent to $n$ in $\sigma$. If not, then inserting $n+1$ anywhere to the left of $n-1$ contains an occurrence of $312$. Similarly, we must insert $n+1$ left-adjacent to $n$ or at the end in $\sigma'$. Inserting $n+1$ anywhere to the right of $n-1$ contains an occurrence of $231$. 
        
        We now show that $(\sigma_L (n-1) (n+1) n, \sigma_L' (n+1) n (n-1)  \sigma_R')$ cannot avoid these patterns. 
        Similar to Figure \ref{fig1} above, we can subdivide $\sigma_L$ into $\pi_L$ and $\pi_R$, where $\pi_L$ is of the same size as $\sigma_L'$, and $\pi_R(n-1)$ is consecutively increasing.
        
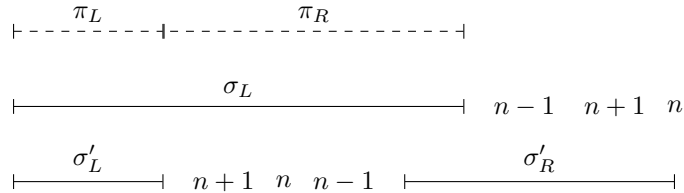
\begin{center}
\begin{tikzpicture}[coo/.style={coordinate}]
	\path (0,0) coordinate (x1) --++ (2,0) coordinate (x2)
			--++ (0.8,0) coordinate (x3)
			--++ (0.8,0) coordinate (x4)
			--++ (0.8,0) coordinate (x5)
			--++ (0.8,0) coordinate (x6)
			--++ (0.8,0) coordinate (x7)
			--++ (0.8,0) coordinate (x8)
			--++ (1.2,0) coordinate (x9)
			--++ (0.8,0) coordinate (xx);
	\foreach \i in {1,2,3}
		\coordinate[] (L\i) at (0,-\i);

	\draw[|-|] (L2-|x1) -- (L2-|x7) node[above,pos=0.5] {$\sigma_L$};
	\node at (L2-|x8) {$n-1$};
	\node at (L2-|x9) {$n+1$};
	\node at (L2-|xx) {$n$};
		 
	\draw[|-|] (L3-|x1) -- (L3-|x2) node[above,pos=0.5] {$\sigma_L'$};
	\node at (L3-|x5) {$n-1$};
	\node at (L3-|x4) {$n$};
	\node at (L3-|x3) {$n+1$};
	\draw[|-|] (L3-|x6) -- (L3-|xx) node[above,pos=0.5] {$\sigma_R'$}; 
	
	\draw[|-|,dashed] (L1-|x1) -- (L1-|x2) node[above,pos=0.5] {$\pi_L$};
	\draw[|-|,dashed] (L1-|x7) -- (L1-|x2) node[above,pos=0.5] {$\pi_R$};

\end{tikzpicture}
\end{center}
        
        Then the projection is of the form $(\sigma_L' \circ \pi_L^{-1}) \pi (r+2) r (r+1)$, where $r$ is the minimal element of $(n-1)\sigma_R'$ and $\pi$ is a subpermutation. This contains an occurrence of $312$.
        
        A similar calculation shows that $\proj(\sigma_L (n-1) (n+1) n, \sigma_L' n (n-1)  \sigma_R' (n+1))$ is of the form $\pi (r+1) (n+1) r$ for a subpermutation $\pi$, which contains an occurrence of $231$. Hence, each instance of $\boldsymbol{\sigma}$ in this case contributes two new elements in $S_{n+1}^2(231,312)$.
    
        \item $(\sigma, \sigma')= (\sigma_L n (n-1) \sigma_R, \sigma_L' (n-1) n)$.
        
        Using a similar argument to the previous case, $n(n-1)\sigma_R$ must be consecutively decreasing. As in the previous cases, $(\sigma_L n (n-1) \sigma_R (n+1), \sigma_L' (n-1) n (n+1))$ avoids $231$ and $312$. Moreover, $(\sigma_L (n+1) n (n-1) \sigma_R, \sigma_L' (n-1) n (n+1))$ also avoids these patterns. To see this, we first evaluate the projection of $(\sigma_L n (n-1) \sigma_R, \sigma_L' (n-1) n)$, which we will call $\gamma$:
        
        \begin{center}
    \begin{tikzpicture}[coo/.style={coordinate}]
	\path (0,0) coordinate (x1) --++ (2,0) coordinate (x2)
			--++ (0.6,0) coordinate (x3)
			--++ (1.0,0) coordinate (x4)
			--++ (0.8,0) coordinate (x5)
			--++ (2,0) coordinate (x6)
			--++ (0.8,0) coordinate (x7)
			--++ (0.8,0) coordinate (xx);
	\foreach \i in {1,2,3}
		\coordinate[] (L\i) at (0,-\i);
		
	\draw[|-|] (L1-|x1) -- (L1-|x2) node[above,pos=0.5] {$\sigma_L$};
	\node at (L1-|x3) {$n$};
	\node at (L1-|x4) {$n-1$};
	\draw[|-|] (L1-|x5) -- (L1-|xx) node[above,pos=0.5] {$\sigma_R$};	
		 
	\draw[|-|] (L2-|x1) -- (L2-|x6) node[above,pos=0.5] {$\sigma_L'$};
	\node at (L2-|x7) {$n-1$};
	\node at (L2-|xx) {$n$};
	
	\draw[|-|,dashed] (L3-|x1) -- (L3-|x2) node[above,pos=0.5] {$\pi_L$};
	\draw[|-|,dashed] (L3-|x6) -- (L3-|x2) node[above,pos=0.5] {$\pi_R$};
	\node at (L3-|x4) { };

\end{tikzpicture}
\end{center}
        
        Since $\gamma$ is of the form $(\pi_L \circ \sigma_L^{-1}) n (n-1) \mathrm{rev}(\pi_R)$, we conclude that $n(n-1)\mathrm{rev}(\pi_R)$ must be consecutively decreasing to avoid occurrences of $231$ and $312$. We now evaluate the projection of $(\sigma_L (n+1) n (n-1) \sigma_R, \sigma_L' (n-1) n (n+1))$:
        
        \begin{center}
    \begin{tikzpicture}[coo/.style={coordinate}]
	\path (0,0) coordinate (x1) --++ (2,0) coordinate (x2)
			--++ (0.8,0) coordinate (x3)
			--++ (1.0,0) coordinate (x4)
			--++ (1.0,0) coordinate (x5)
			--++ (0.8,0) coordinate (x6)
			--++ (0.6,0) coordinate (x7)
			--++ (0.8,0) coordinate (x8)
			--++ (1.0,0) coordinate (x9)
			--++ (1.0,0) coordinate (xx);
	\foreach \i in {1,2,3}
		\coordinate[] (L\i) at (0,-\i);

	\draw[|-|] (L1-|x1) -- (L1-|x2) node[above,pos=0.5] {$\sigma_L$};
	\node at (L1-|x3) {$n+1$};
	\node at (L1-|x4) {$n$};
	\node at (L1-|x5) {$n-1$};
	\draw[|-|] (L1-|x6) -- (L1-|xx) node[above,pos=0.5] {$\sigma_R$};	
		 
	\draw[|-|] (L2-|x1) -- (L2-|x7) node[above,pos=0.5] {$\sigma_L'$};
	\node at (L2-|x8) {$n-1$};
	\node at (L2-|x9) {$n$};
	\node at (L2-|xx) {$n+1$};
	
	\draw[|-|,dashed] (L3-|x1) -- (L3-|x2) node[above,pos=0.5] {$\pi_L$};
	\draw[|-|,dashed] (L3-|x7) -- (L3-|x2) node[above,pos=0.5] {$\pi_R$};
	\node at (L3-|x8) { };

\end{tikzpicture}
\end{center}
        
        Now $\proj(\sigma_L (n+1) n (n-1) \sigma_R, \sigma_L' (n-1) n (n+1))$ is $(\pi_L \circ \sigma_L^{-1}) (n+1) n (n-1) \mathrm{rev}(\pi_R)$, which also avoids $231$ and $312$, as shown in Figure \ref{fig:231,132}.

        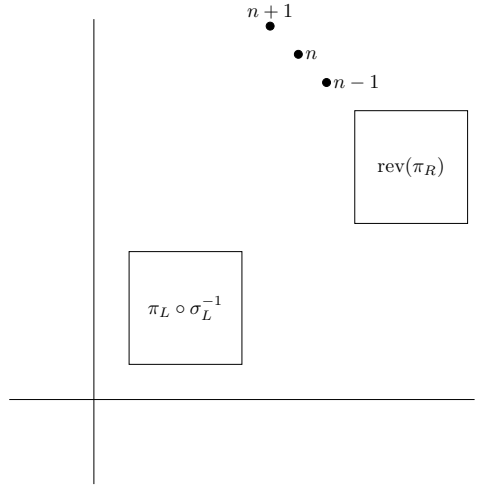
\begin{figure}[h]
    \centering
    \scalebox{0.75}{
    \begin{tikzpicture}
        \foreach [evaluate=\i as \x using int(\i-1)]\i in {0,1,...,10}
        {
            \foreach [evaluate=\j as \y using int(\j-1)] \j in {0,1,...,10}
            {
                \node at (\i,\j)[name=perm-\x-\y,]{};
            }
        }

        \draw ([xshift=-5mm]perm-1--1.south west)--([xshift=-5mm]perm-1-7.north west);
        \draw ([yshift=5mm]perm--1-0.south west)--([yshift=5mm]perm-7-0.south east);

        \draw[draw=black] (2,2) rectangle ++(2,2);
        \node at (3,3) {$\pi_L \circ \sigma_L^{-1}$};

        \filldraw [black] (4.5,8) circle (2pt) node[anchor=south]{$n+1$};

        \filldraw [black] (5,7.5) circle (2pt) node[anchor=west]{$n$};

        \filldraw [black] (5.5,7) circle (2pt) node[anchor=west]{$n-1$};

        \draw[draw=black] (6,4.5) rectangle ++(2,2);
        \node at (7,5.5) {$\rev(\pi_R)$};

\end{tikzpicture}
}
    \caption{An illustration of the projection of $(\sigma_L (n+1) n (n-1) \sigma_R, \sigma_L' (n-1) n (n+1))$.}
    \label{fig:231,132}
\end{figure}
        
        We show that inserting $n+1$ anywhere else in $\boldsymbol{\sigma}$ cannot produce a $3$-permutation that avoids $231$ and $312$. Using similar arguments to the previous case, it is sufficient to show that both $(\sigma_L (n+1) n (n-1) \sigma_R, \sigma_L' (n-1) (n+1) n)$ and $(\sigma_L n (n-1) \sigma_R (n+1), \sigma_L' (n-1) (n+1) n)$ contain an occurrence of $231$ or $312$.
        
        For the former $3$-permutation, we take $\proj(\sigma_L (n+1) n (n-1) \sigma_R, \sigma_L' (n-1) (n+1) n)$:
        
    \begin{center}
    \begin{tikzpicture}[coo/.style={coordinate}]
	\path (0,0) coordinate (x1) --++ (2,0) coordinate (x2)
			--++ (0.8,0) coordinate (x3)
			--++ (1.0,0) coordinate (x4)
			--++ (1.0,0) coordinate (x5)
			--++ (0.8,0) coordinate (x6)
			--++ (0.6,0) coordinate (x7)
			--++ (0.8,0) coordinate (x8)
			--++ (1.2,0) coordinate (x9)
			--++ (1.0,0) coordinate (xx);
	\foreach \i in {1,2,3}
		\coordinate[] (L\i) at (0,-\i);

	\draw[|-|] (L1-|x1) -- (L1-|x2) node[above,pos=0.5] {$\sigma_L$};
	\node at (L1-|x3) {$n+1$};
	\node at (L1-|x4) {$n$};
	\node at (L1-|x5) {$n-1$};
	\draw[|-|] (L1-|x6) -- (L1-|xx) node[above,pos=0.5] {$\sigma_R$};	
		 
	\draw[|-|] (L2-|x1) -- (L2-|x7) node[above,pos=0.5] {$\sigma_L'$};
	\node at (L2-|x8) {$n-1$};
	\node at (L2-|x9) {$n+1$};
	\node at (L2-|xx) {$n$};
	
	\draw[|-|,dashed] (L3-|x1) -- (L3-|x2) node[above,pos=0.5] {$\pi_L$};
	\draw[|-|,dashed] (L3-|x7) -- (L3-|x2) node[above,pos=0.5] {$\pi_R$};
	\node at (L3-|x8) { };

\end{tikzpicture}
\end{center}
        
        The projection is of the form $(\pi_L \circ \sigma_L^{-1}) n (n+1) (n-1) \mathrm{rev}(\pi_R)$, which contains an occurrence of $231$.
        
        For the latter $3$-permutation, a similar argument shows that this projection contains an occurrence of $312$. Hence, each instance of $\boldsymbol{\sigma}$ in this case contributes $2$ new elements in $S_{n+1}^2(231,312)$.
    
        \item $(\sigma, \sigma')= (\sigma_L n (n-1) \sigma_R, \sigma_L' n (n-1) \sigma_R')$.
        
        Then $n(n-1)\sigma_R$ and $n(n-1)\sigma_R'$ must be consecutively decreasing. We claim $|\sigma_R| = |\sigma_R'|$. For the sake of contradiction, suppose that $|\sigma_R'|>|\sigma_R|$. Then the $3$-permutations are of the following form:
        
    \begin{center}
    \begin{tikzpicture}[coo/.style={coordinate}]
	\path (0,0) coordinate (x1) --++ (2,0) coordinate (x2)
			--++ (0.6,0) coordinate (x3)
			--++ (1,0) coordinate (x4)
			--++ (0.8,0) coordinate (x5)
			--++ (2,0) coordinate (x6)
			--++ (0.6,0) coordinate (x7)
			--++ (1,0) coordinate (x8)
			--++ (0.8,0) coordinate (x9)
			--++ (2.5,0) coordinate (xx);
	\foreach \i in {1,2,3}
		\coordinate[] (L\i) at (0,-\i);

	\draw[|-|] (L1-|x1) -- (L1-|x6) node[above,pos=0.5] {$\sigma_L$};
	\node at (L1-|x7) {$n$};
	\node at (L1-|x8) {$n-1$};
	\draw[|-|] (L1-|x9) -- (L1-|xx) node[above,pos=0.5] {$\sigma_R$};	
		 
	\draw[|-|] (L2-|x1) -- (L2-|x2) node[above,pos=0.5] {$\sigma_L'$};
	\node at (L2-|x3) {$n$};
	\node at (L2-|x4) {$n-1$};
	\draw[|-|] (L2-|x5) -- (L2-|xx) node[above,pos=0.5] {$\sigma_R'$};

\end{tikzpicture}
\end{center}
        
        The projection, which we will call $\rho$, is of the form $\pi_1 n \pi_2 r \pi_3 (r+c)$, where $r$ is the minimal element of $\sigma_R'$, $c$ is some positive integer, and $\pi_1$, $\pi_2$, $\pi_3$ are subpermutations. Figure \ref{fig:231,132 p2} illustrates this projection. Hence, $\rho$ contains $312$, a contradiction. 

               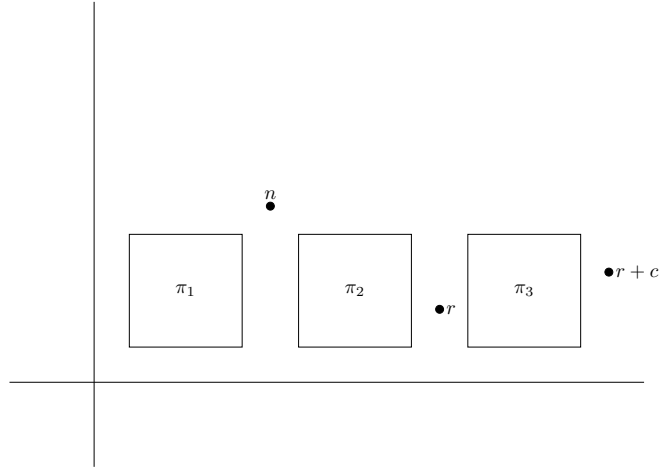
\begin{figure}[h]
    \centering
    \scalebox{0.75}{
    \begin{tikzpicture}
        \foreach [evaluate=\i as \x using int(\i-1)]\i in {0,1,...,11}
        {
            \foreach [evaluate=\j as \y using int(\j-1)] \j in {0,1,...,11}
            {
                \node at (\i,\j)[name=perm-\x-\y,]{};
            }
        }

        \draw ([xshift=-5mm]perm-1--1.south west)--([xshift=-5mm]perm-1-7.north west);
        \draw ([yshift=5mm]perm--1-0.south west)--([yshift=5mm]perm-10-0.south east);

        \draw[draw=black] (2,2) rectangle ++(2,2);
        \node at (3,3) {$\pi_1$};

        \filldraw [black] (4.5,4.5) circle (2pt) node[anchor=south]{$n$};

        \draw[draw=black] (5,2) rectangle ++(2,2);
        \node at (6,3) {$\pi_2$};

        \filldraw [black] (7.5,2.67) circle (2pt) node[anchor=west]{$r$};

        \draw[draw=black] (8,2) rectangle ++(2,2);
        \node at (9,3) {$\pi_3$};

        \filldraw [black] (10.5,3.33) circle (2pt) node[anchor=west]{$r+c$};

\end{tikzpicture}
}
    \caption{An illustration of the projection of $(\sigma_L n (n-1) \sigma_R, \sigma_L' n (n-1) \sigma_R')$ when $|\sigma_R'|>|\sigma_R|$.}
    \label{fig:231,132 p2}
\end{figure}
        
        A similar argument holds for $|\sigma_R'|<|\sigma_R|$. Hence, the two permutations must be of the same size. Moreover, since both $n(n-1)\sigma_R$ and $n(n-1)\sigma_R'$ are consecutively decreasing, then we have $\sigma_R = \sigma_R'$.
        
        We immediately see that $(\sigma_L n (n-1) \sigma_R (n+1), \sigma_L' n (n-1) \sigma_R' (n+1))$ is in $S_{n+1}^2 (231,312)$. Moreover, note that the projection of $(\sigma_L (n+1) n (n-1) \sigma_R, \sigma_L' (n+1) n (n-1) \sigma_R')$ is of the form $(\sigma_L' \circ \sigma_L^{-1}) (\sigma_R' \circ \sigma_R^{-1}) (n-1) n (n+1)$, and hence, $(\sigma_L (n+1) n (n-1) \sigma_R, \sigma_L' (n+1) n (n-1) \sigma_R')$ also avoids $231$ and $312$.
        
        Now we show that inserting the maximal element $n+1$ anywhere else cannot avoid $231$ and $312$. In fact, $n+1$ can only be inserted either at the end of $\sigma$ and $\sigma'$ or left-adjacent to $n$. If $n+1$ is inserted anywhere in $\sigma_L$ or $\sigma_L'$, then there would be an occurrence of $312$. If $n+1$ is inserted anywhere to the right of $n$ and not at the end of the permutation, then there would be an occurrence of $231$. 
        
        We show that both the $3$-permutations $(\sigma_L (n+1) n (n-1) \sigma_R, \sigma_L' n (n-1) \sigma_R' (n+1))$ and $(\sigma_L n (n-1) \sigma_R (n+1), \sigma_L' (n+1) n (n-1) \sigma_R')$ contain an occurrence of either $231$ or $312$.
        
        For the first $3$-permutation, the projection looks as follows:
        
            \begin{center}
    \begin{tikzpicture}[coo/.style={coordinate}]
	\path (0,0) coordinate (x1) --++ (2,0) coordinate (x2)
			--++ (0.8,0) coordinate (x3)
			--++ (0.8,0) coordinate (x4)
			--++ (0.8,0) coordinate (x5)
			--++ (0.8,0) coordinate (x6)
			--++ (0.8,0) coordinate (x7)
			--++ (0.8,0) coordinate (xx);
	\foreach \i in {1,2,3}
		\coordinate[] (L\i) at (0,-\i);

	\draw[|-|] (L1-|x1) -- (L1-|x2) node[above,pos=0.5] {$\sigma_L$};
	\node at (L1-|x5) {$n-1$};
	\node at (L1-|x4) {$n$};
	\node at (L1-|x3) {$n+1$};
	\draw[|-|] (L1-|x6) -- (L1-|xx) node[above,pos=0.5] {$\sigma_R$}; 
		 
	\draw[|-|] (L2-|x1) -- (L2-|x2) node[above,pos=0.5] {$\sigma_L'$};
	\node at (L2-|x4) {$n-1$};
	\node at (L2-|x3) {$n$};
	\node at (L2-|xx) {$n+1$};
	\draw[|-|] (L2-|x5) -- (L2-|x7) node[above,pos=0.5] {$\sigma_R'$};

\end{tikzpicture}
\end{center}
        
        Evaluating the projection gives the form $(\sigma_L' \circ \sigma_L^{-1})(n+1) \pi (n-1)(n)$ for a subpermutation~$\pi$, which contains $312$.
        
        A similar argument shows that $\proj(\sigma_L n (n-1) \sigma_R (n+1), \sigma_L' (n+1) n (n-1) \sigma_R')$ contains an occurrence of $231$. Therefore, each instance of $\boldsymbol{\sigma}$ in this case contributes $2$ new elements in $S_{n+1}^2(231,312)$. 
    \end{enumerate}

Now we show that $3$-permutations avoiding $231$ and $312$ must be of one of the forms above. We have only one form to consider, where exactly one of $\sigma_R$, $\sigma_R'$ is empty. Let $\sigma_R$ be empty and $\sigma_R'$ be nonempty. In particular, $(\sigma, \sigma')= (\sigma_L n, \sigma_L' n \sigma_R')$. 
    
    Now $n-1$ must be adjacent to $n$ in $\sigma'$. If $\sigma' = \sigma_L' (n-1) n \sigma_R'$, then $\sigma_R'$ must be empty to avoid an occurrence of $231$. Then Case 1 covers this. If $\sigma' = \sigma_L' n (n-1) \sigma_R'$, then we show that $n-1$ is adjacent to $n$ in $\sigma$. Suppose, for the sake of contradiction, that this is not the case. First, $n(n-1)\sigma_R'$ must be consecutively decreasing. Taking the projection $\proj (\sigma, \sigma')$, we conclude that it is of the form $\pi_L n \pi_R k r$, where $\pi_L$ and $\pi_R$ are subpermutations, $k \neq r+1$, and $r$ is the minimal element in $\sigma_R'$ (note that if $\sigma_R'$ is empty, then the projection contains the sequence $nk (n-1)$, which is a $312$ pattern). Now we consider where the element $r+1$ is in the permutation. If $r+1$ is in $\pi_L$, then there is an occurrence of $231$. If $r+1$ is in $\pi_R$, then if $k>r+1$, there is an occurrence of $231$ and if $k<r+1$, there is an occurrence of $312$. Hence, $n-1$ must be adjacent to $n$ in $\sigma$, and a similar argument from Case 2 covers this case. A similar argument also holds for nonempty $\sigma_R$ and empty $\sigma_R'$.

Now we show that for $(\sigma, \sigma') \in S_n^2 \setminus S_n^2(231,312)$, we cannot obtain an element in $ S_{n+1}^2(231,312)$ by inserting the maximal element $n+1$ anywhere in $\sigma$ and $\sigma'$. We will assume that $\sigma$ and $\sigma'$ avoid these patterns but $\sigma' \circ \sigma^{-1}$ does not. We iterate through the same cases as above:

\begin{enumerate}
    \item $(\sigma, \sigma') = (\sigma_L n, \sigma_L' n)$.
    
    It's straightforward to check that the projections of the following $3$-permutations contain $\sigma' \circ \sigma^{-1}$: $(\sigma_L n (n+1), \sigma_L' n (n+1))$, ${(\sigma_L n (n+1), \sigma_L' (n+1)n)}$, ${(\sigma_L (n+1) n, \sigma_L' n (n+1))},$ and $(\sigma_L (n+1)n, \sigma_L' (n+1)n)$.

    \item $(\sigma, \sigma')= (\sigma_L (n-1) n, \sigma_L' n (n-1)  \sigma_R')$.
    
    It is clear that $(\sigma_L (n-1) n (n+1), \sigma_L' n (n-1)  \sigma_R' (n+1))$ contains $\sigma' \circ \sigma^{-1}$. Now we consider $(\sigma_L (n-1) n (n+1), \sigma_L' (n+1) n (n-1)  \sigma_R')$. Call its projection $\rho$. Using the notation in Case 2 above, note that $\pi_R (n-1) n$ must be consecutively increasing. Else $\rho$ contains an occurrence of $231$ if $\pi_R (n-1) n$ contains $21$, and $\rho$ contains an occurrence of $312$ if $\pi_R (n-1) n$ is increasing but not consecutive. As such, $\rho$ must be of the form $(\sigma_L' \circ \pi_L^{-1}) (n+1) n (n-1) \sigma_R'$ and $\sigma' \circ \sigma^{-1}$ is of the form $(\sigma_L' \circ \pi_L^{-1}) n (n-1) \sigma_R'$. It is clear that $\rho$ contains $\sigma' \circ \sigma^{-1}$. Lastly, recall that inserting $n+1$ anywhere else in $\sigma$ and $\sigma'$ must result in a $3$-permutation that contains an occurrence of $231$ or $312$.

    \item $(\sigma, \sigma')= (\sigma_L n (n-1) \sigma_R, \sigma_L' (n-1) n)$.
    
    Using a similar logic to the previous case, $(\sigma_L n (n-1) \sigma_R (n+1), \sigma_L' (n-1) n (n+1))$ and $(\sigma_L (n+1) n (n-1) \sigma_R, \sigma_L' (n-1) n (n+1))$ contain $\sigma' \circ \sigma^{-1}$, and inserting $n+1$ anywhere else must result in a $3$-permutation that contains either $231$ or $312$.

    \item $(\sigma, \sigma')= (\sigma_L n (n-1) \sigma_R, \sigma_L' n (n-1) \sigma_R')$.
    
    Our logic in Case 4 above shows that $(\sigma_L n (n-1) \sigma_R (n+1), \sigma_L' n (n-1) \sigma_R' (n+1))$ and $(\sigma_L (n+1) n (n-1) \sigma_R, \sigma_L' (n+1) n (n-1) \sigma_R')$ contain $\sigma' \circ \sigma^{-1}$, and inserting $n+1$ anywhere else must result in a $3$-permutation that contains these patterns.
\end{enumerate}

However, we must also consider when exactly one of $\sigma_R$, $\sigma_R'$ is empty, but similar logic as presented in our proof shows either $(\sigma, \sigma')$ belongs to one of the cases we iterated above or inserting $n+1$ anywhere in $\sigma$ and $\sigma'$ results in a $3$-permutation that contains either pattern.

Therefore, we see that for every $3$-permutation $\boldsymbol{\sigma} = (\sigma, \sigma')$ in $S_n^2(231,312)$, inserting the maximal element $n+1$ onto the end of both $\sigma$ and $\sigma'$ always yields a $3$-permutation in $S_{n+1}^2(231,312)$; moreover, inserting the maximal element such that the relative positions of the two largest elements in both $\sigma$ and $\sigma'$ are preserved also always yields another $3$-permutation. This contributes $2a_n$ different $3$-permutations to $S_{n+1}^2(231,312)$. In the case that $\boldsymbol{\sigma}$ is of the form in Case 1 (where $\sigma$, $\sigma'$ each end with the maximal element $n$), each $\boldsymbol{\sigma}$ can construct two elements in $S_{n+1}^2(231,312)$ in addition to the elements generated above, and this case contributes $2a_{n-1}$ additional elements in $S_{n+1}^2(231,312)$. We have that \begin{align*}
    a_{n+1} = 2a_n + 2a_{n-1}. & \qedhere
\end{align*}
\end{proof}

\begin{theorem}\label{231,321}
Let $a_n = |S^2_n(231,321)|$. Then $a_n$ follows the formula $a_{n+1} = 4 \cdot 3^{n-1}$ (where $a_1 = 1$), which corresponds to the OEIS sequence \href{http://oeis.org/A003946}{A003946}.
\end{theorem}
\begin{proof}

Since the pair of patterns $\{231,321\}$ is trivially Wilf-equivalent to the pair of patterns $\{312,321\}$, we will show the the above recurrence relation for when $a_n = |S^2_n(312,321)|$.

Let $\boldsymbol{\sigma} = (\sigma, \sigma') \in S_n^2(312,321)$ and let $\boldsymbol{\sigma}$ be of the form $(\sigma_L n \sigma_R, \sigma_L' n \sigma_R')$. Note that $\sigma_R$ and~$\sigma_R'$ either contain one element or are empty.

We insert the maximal element $n+1$ to the permutation. The element $n+1$ must be inserted at the end or second-to-end in both $\sigma$ and $\sigma'$. We have the following~cases:

\begin{enumerate}
    \item $(\sigma, \sigma') = (\sigma_L n, \sigma_L' n)$.
    
    We can see that $(\sigma_L n (n+1), \sigma_L' n (n+1))$, $(\sigma_L n (n+1), \sigma_L' (n+1) n)$, $(\sigma_L (n+1) n, \sigma_L' n (n+1))$, and $(\sigma_L (n+1) n, \sigma_L' (n+1) n)$ all avoid $312$ and $321$.
    
    A $3$-permutation $\boldsymbol{\sigma}$ in this case constructs $4$ distinct $3$-permutations in $S_{n+1}^2 (312,321)$.
    
    \item $(\sigma, \sigma') = (\sigma_L n, \sigma_L' n r')$ for some integer $r'$.
    
    Similarly to the previous case, we see that the $3$-permutations $(\sigma_L n (n+1), \sigma_L' n r' (n+1))$, $(\sigma_L n (n+1), \sigma_L' n (n+1) r')$, $(\sigma_L (n+1) n, \sigma_L' n r' (n+1))$, and $(\sigma_L (n+1) n, \sigma_L' n (n+1) r')$ all avoid $312$ and $321$.
    
    A $3$-permutation $\boldsymbol{\sigma}$ in this case also constructs $4$ distinct $3$-permutations in $S_{n+1}^2 (312,321)$.

    \item $(\sigma, \sigma') = (\sigma_L n r, \sigma_L' n)$ for some integer $r$.
    
    Appending the maximal element $n+1$ at the end of $\sigma$ and $\sigma'$ still avoids $312$ and $321$. In particular, $(\sigma_L n r (n+1), \sigma_L' n (n+1))$, as well as $(\sigma_L n (n+1) r, \sigma_R' (n+1) n)$, avoids these patterns.
    
    Now we show that inserting $n+1$ anywhere else must contain these patterns. In particular, note that $(\sigma_L n r (n+1), \sigma_L' (n+1) n)$ and $(\sigma_L n (n+1) r, \sigma_L' n (n+1))$ cannot avoid $312$ and $321$. For both $3$-permutations, the projection is of the form $\pi_L (n+1) \pi_R n$ for subpermutations $\pi_L$ and $\pi_R$ (where $\pi_R$ is nonempty). This contains an instance of $312$.
    
    Hence, $3$-permutations $\boldsymbol{\sigma}$ in this case produce two different elements in $S_{n+1}^2 (312,321)$.
    
    \item $(\sigma, \sigma') = (\sigma_L n r, \sigma_L' n r')$ for integers $r$, $r'$.
    
    As in the previous cases, note that the $3$-permutation $(\sigma_L n r (n+1), \sigma_L' n r' (n+1))$, as well as $(\sigma_L n (n+1) r, \sigma_L' n (n+1) r')$, avoids $312$ and $321$. 
    
    Now we show that inserting $n+1$ anywhere else in $\boldsymbol{\sigma}$ cannot avoid $312$ and $321$. In particular, we show that $(\sigma_L n r (n+1), \sigma_L' n (n+1) r')$ and $(\sigma_L n (n+1) r, \sigma_L' n r' (n+1))$ cannot avoid $312$ and $321$.
    
    For both $3$-permutations, the projection is $\pi_L (n+1) \pi_R n r$ for some subpermutations $\pi_L, \pi_R$. This contains an instance of $321$. And hence, $3$-permutations $\boldsymbol{\sigma}$ in this case construct $2$ distinct elements in $S_{n+1}^2 (312,321)$.
    
\end{enumerate}

Now we show that for $(\sigma,\sigma') \in S_n^2 \setminus S_{n}^2(312,321)$, we cannot obtain an element in $ S_{n+1}^2(312,321)$ by inserting the maximal element $n+1$ anywhere in $\sigma$ and $\sigma'$. So let $\sigma' \circ \sigma^{-1}$ contain an instance of $312$ or $321$. Then $(\sigma, \sigma')$ must be one of the cases above, and it is clear that the projections of the $3$-permutations we obtained by inserting $n+1$ contain $\sigma' \circ \sigma^{-1}$, which must also contain an instance of these patterns. Hence, for $(\sigma,\sigma') \in S_n^2 \setminus S_{n}^2(312,321)$, we cannot obtain an element in $ S_{n+1}^2(312,321)$ by inserting the maximal element $n+1$ anywhere in $\sigma$ and $\sigma'$.

Now we claim that in $S_n^2(312,321)$, exactly half of the elements $\boldsymbol{\sigma} = (\sigma, \sigma')$ satisfy $\sigma(n) = n$. The base case can be seen in $S_2^2(312,321)$. Then for our inductive step let us assume that this is the case for $S_{n-1}^2(312,321)$. We wish to show that this is true for $S_{n}^2(312,321)$. In each case above, exactly half of the $3$-permutations constructed have the property $\sigma(n)= n$ and the other half satisfy $\sigma(n) \neq n$, and via induction, exactly half of the elements in $S_n^2(312,321)$ satisfy $\sigma(n) = n$.

Note that if $\sigma(n) = n$, we are in Case 1 or Case 2, which contribute $4$ elements in $S_{n+1}^2(312,321)$. When $\sigma(n) \neq n$, we are in Case 3 or Case 4, which contribute $2$ elements in $S_{n+1}^2(312,321)$.

Thus, we conclude that $$a_{n+1} = \frac{a_n}{2}\cdot 4 + \frac{a_n}{2} \cdot 2 = 3a_n.$$

We can see that $a_2 = 4$, and we conclude that \begin{align*}
    a_{n+1} = 4 \cdot 3^{n-1}. & \qedhere
\end{align*}
\end{proof}

This allows us to prove all the conjectures Bonichon and Morel \cite{bonichon2022baxter} have made in regard to $3$-permutations avoiding two patterns of size $3$. However, there is one class of $3$-permutations that have yet to be classified, which we now enumerate. We begin with an observation.

\begin{observation}\label{inverselemma}
Let $\sigma$ be a permutation and $\pi$ be an involution. Then $\sigma$ avoids $\pi$ if and only if $\sigma^{-1}$ avoids $\pi$. 
\end{observation}

Since $132$ and $213$ are both involutions, $\sigma$ avoids $132$ if and only if $\sigma^{-1}$ avoids $132$. The same reasoning holds for the pattern $213$. We then have a corollary:

\begin{corollary}\label{coro}
    Let $\pi$ be an involution. Then the $3$-permutation $(\sigma, \sigma')$ avoids $\pi$ if and only if the $3$-permutation $(\sigma',\sigma)$ avoids $\pi$.
\end{corollary}

This is due to the fact that $\sigma' \circ \sigma^{-1}$ avoids $\pi$ if and only if $\sigma \circ (\sigma')^{-1}$ avoids $\pi$.

\begin{theorem}\label{132,213}
Let $a_n = |S_n^2(132,213)|$. Then $a_n$ satisfies the recurrence  $$a_{n+1} = a_n + 3 \cdot 2^{n-1} +2(n-1)$$ with the initial term $a_1 = 1$. This corresponds to the OEIS sequence \href{http://oeis.org/A356728}{A356728}.
\end{theorem}
\begin{proof}

Let $\boldsymbol{\sigma} = (\sigma, \sigma') \in S_n^2(132,213)$ and let $\boldsymbol{\sigma}$ be of the form $(\sigma_L n \sigma_R, \sigma_L' n \sigma_R')$.

Note that $\sigma_L n$ and $\sigma_L' n$ are increasing; otherwise the permutation would contain an occurrence of $213$. Moreover, they must be consecutively increasing; otherwise we would have an occurrence of $132$.

Adding the maximal element $n+1$ right-adjacent to $n$ in both $\sigma$ and $\sigma'$ always produces a $3$-permutation in $S_{n+1}^2(132,213)$. To see this, suppose that $|\sigma_L| > |\sigma_L'|$. Then the projection  would look as follows:

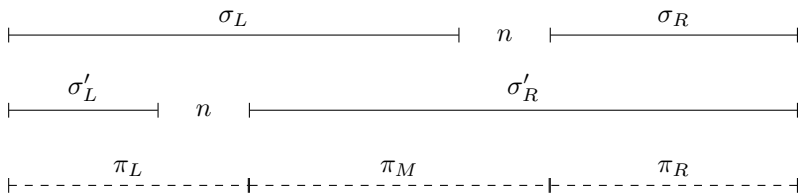
\begin{center}
    \begin{tikzpicture}[coo/.style={coordinate}]
	\path (0,0) coordinate (x1) --++ (2,0) coordinate (x2)
			--++ (0.6,0) coordinate (x3)
			--++ (0.6,0) coordinate (x4)
			--++ (0.8,0) coordinate (x5)
			--++ (2,0) coordinate (x6)
			--++ (0.6,0) coordinate (x7)
			--++ (0.6,0) coordinate (x8)
			--++ (0.8,0) coordinate (x9)
			--++ (2.5,0) coordinate (xx);
	\foreach \i in {1,2,3}
		\coordinate[] (L\i) at (0,-\i);

	\draw[|-|] (L1-|x1) -- (L1-|x6) node[above,pos=0.5] {$\sigma_L$};
	\node at (L1-|x7) {$n$};
	\draw[|-|] (L1-|x8) -- (L1-|xx) node[above,pos=0.5] {$\sigma_R$};	
		 
	\draw[|-|] (L2-|x1) -- (L2-|x2) node[above,pos=0.5] {$\sigma_L'$};
	\node at (L2-|x3) {$n$};
	\draw[|-|] (L2-|x4) -- (L2-|xx) node[above,pos=0.5] {$\sigma_R'$}; 
	
	\draw[|-|,dashed] (L3-|x1) -- (L3-|x4) node[above,pos=0.5] {$\pi_L$};
	\draw[|-|,dashed] (L3-|xx) -- (L3-|x8) node[above,pos=0.5] {$\pi_R$};
	\draw[|-|,dashed] (L3-|x4) -- (L3-|x8) node[above,pos=0.5] {$\pi_M$};

\end{tikzpicture}
\end{center}

This has the form $(\pi_R \circ \sigma_R^{-1}) \pi_L \pi_M$, where $\pi_L$ is consecutively increasing and ends with $n$. Figure \ref{fig:132,213} illustrates this projection. This projection must avoid $132$ and $213$.

    \begin{figure}[h]
    \centering
    \scalebox{0.75}{
    \begin{tikzpicture}
        \foreach [evaluate=\i as \x using int(\i-1)]\i in {0,1,...,11}
        {
            \foreach [evaluate=\j as \y using int(\j-1)] \j in {0,1,...,11}
            {
                \node at (\i,\j)[name=perm-\x-\y,]{};
            }
        }

        \draw ([xshift=-5mm]perm-1--1.south west)--([xshift=-5mm]perm-1-7.north west);
        \draw ([yshift=5mm]perm--1-0.south west)--([yshift=5mm]perm-10-0.south east);

        \draw[draw=black] (2,2) rectangle ++(2,2);
        \node at (3,3) {$\pi_R \circ \sigma_R^{-1}$};

        \draw[draw=black] (5,5) rectangle ++(2,2);
        \node at (6,6) {$\pi_L$};

        \draw[draw=black] (8,2) rectangle ++(2,2);
        \node at (9,3) {$\pi_M$};

\end{tikzpicture}
}
    \caption{An illustration of the projection of $(\sigma_L n \sigma_R, \sigma_L' n \sigma_R')$ when $|\sigma_L|>|\sigma_L'|$.}
    \label{fig:132,213}
\end{figure}

Now consider $(\sigma_1, \sigma_2) = (\sigma_L n (n+1) \sigma_R, \sigma_L' n (n+1) \sigma_R')$. The projection would look as follows:

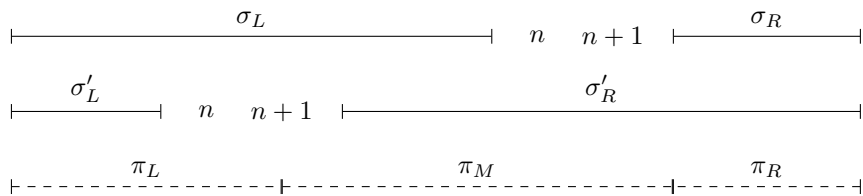
\begin{center}
    \begin{tikzpicture}[coo/.style={coordinate}]
	\path (0,0) coordinate (x1) --++ (2,0) coordinate (x2)
			--++ (0.6,0) coordinate (x3)
			--++ (1,0) coordinate (x4)
			--++ (0.8,0) coordinate (x5)
			--++ (2,0) coordinate (x6)
			--++ (0.6,0) coordinate (x7)
			--++ (1,0) coordinate (x8)
			--++ (0.8,0) coordinate (x9)
			--++ (2.5,0) coordinate (xx);
	\foreach \i in {1,2,3}
		\coordinate[] (L\i) at (0,-\i);

	\draw[|-|] (L1-|x1) -- (L1-|x6) node[above,pos=0.5] {$\sigma_L$};
	\node at (L1-|x7) {$n$};
	\node at (L1-|x8) {$n+1$};
	\draw[|-|] (L1-|x9) -- (L1-|xx) node[above,pos=0.5] {$\sigma_R$};	
		 
	\draw[|-|] (L2-|x1) -- (L2-|x2) node[above,pos=0.5] {$\sigma_L'$};
	\node at (L2-|x3) {$n$};
	\node at (L2-|x4) {$n+1$};
	\draw[|-|] (L2-|x5) -- (L2-|xx) node[above,pos=0.5] {$\sigma_R'$}; 
	
	\draw[|-|,dashed] (L3-|x1) -- (L3-|x4) node[above,pos=0.5] {$\pi_L$};
	\draw[|-|,dashed] (L3-|xx) -- (L3-|x9) node[above,pos=0.5] {$\pi_R$};
	\draw[|-|,dashed] (L3-|x4) -- (L3-|x9) node[above,pos=0.5] {$\pi_M$};

\end{tikzpicture}
\end{center}

This has the form $(\pi_R \circ \sigma_R^{-1}) \pi_L (n+1) \pi_M$, which still avoids $132$ and $213$. The case where $|\sigma_L| = |\sigma_L'|$ follows as well. For the case $|\sigma_L| < |\sigma_L'|$, we utilize Corollary \ref{coro}. Note that due to symmetry, our previous argument implies that the $3$-permutation $(\sigma',\sigma)$ avoids $132$ and $213$, and hence Corollary \ref{coro} states that $(\sigma, \sigma')$ avoids these patterns. 

Hence, appending the maximal element $n+1$ right-adjacent to $n$ in both $\sigma$ and $\sigma'$ contributes $a_{n}$ elements in $S^2_{n+1}(132,213).$

In the following cases, note that $n+1$ must be inserted either at the beginning or right-adjacent to $n$ in $\sigma$ and $\sigma'$. If we insert $n+1$ to the left of $n$ (but not at the beginning), then there is an instance of $132$. Similarly, if we insert $n+1$ to the right of $n$ (but not adjacent to $n$), then there is an instance of $213$. We have the following:

\begin{enumerate}
    \item $\sigma = \sigma'$.
    
    Then $((n+1) \sigma, (n+1) \sigma')$ also avoids $132$ and $213$. Further, note that in the special case when $\sigma = \sigma' = \Id_n$, then $((n+1)\Id_n, \Id_n (n+1))$ and $(\Id_n (n+1), (n+1) \Id_n)$ both avoid $132$ and $213$. Now we show that inserting $n+1$ anywhere else in $\boldsymbol{\sigma}$ does not avoid $132$ and $213$. Specifically, for all other $\sigma$ and $\sigma'$, we show that $((n+1)\sigma, \sigma_L' n (n+1) \sigma_R')$ and $(\sigma_L n (n+1) \sigma_R, (n+1) \sigma')$ cannot avoid $132$ or $213$. To see that the first $3$-permutation cannot avoid $132$ or $213$, note that $\proj((n+1)\sigma, \sigma_L' n (n+1) \sigma_R')$ is of the form $1\pi(n+1)\ell$, where $\ell$ is the first element in $\sigma_L' n$ and $\pi$ is a subpermutation. This contains an occurrence of $132$. A similar argument shows that the projection of the latter $3$-permutation also contains $213$. 
    
    Since Simion and Schmidt \cite{simion1985restricted} showed there are $2^{n-1}$ possible permutations that avoid $132$ and $213$ with size $n$, this contributes an additional $2^{n-1} + 2$ elements in $S_{n+1}^2(132,213)$.
    
    \item $\sigma = \Id_n$ and $\sigma' \neq \Id_n$.
    
    We note that $(\sigma (n+1), (n+1) \sigma')$ avoids $132$ and $213$. In the special case where $\sigma_R'$ is consecutively increasing, then $((n+1) \sigma, \sigma_L' n (n+1) \sigma_R')$ is also an element in $S_{n+1}^2(132,213)$.
    
    Now we show that inserting $n+1$ anywhere else in $\boldsymbol{\sigma}$ cannot avoid $132$ and $213$. We first show that $(\sigma (n+1), \sigma' (n+1))$ and $((n+1)\sigma, (n+1) \sigma')$ cannot avoid $132$ and $213$. Taking the projection of these $3$-permutations evaluates to $\sigma' (n+1)$, which contains an occurrence of $213$ because $\sigma'$ is not the identity and therefore, must contain an occurrence of $21$.
    
    Note that $((n+1) \sigma, \sigma' (n+1))$ cannot avoid these patterns because $\sigma' (n+1)$ contains an occurrence of $213$.
    
    Now let $\sigma_R'$ be decreasing. We wish to show that $((n+1) \sigma, \sigma_L' n (n+1) \sigma_R')$ cannot avoid $132$ and $213$. Since $\sigma_R'$ contains an instance of $21$ and every element in $\sigma_R'$ is smaller than every element in $\sigma_L'$, taking $\proj((n+1) \sigma, \sigma_L' n (n+1) \sigma_R')$ gives an occurrence of $213$.
    
    Since there are $2^{n-1}$ different $\sigma'$ that avoid $132$ and $213$, note that $(\sigma (n+1), (n+1) \sigma')$ contributes $2^{n-1}-1$ different elements to $S_{n+1}^2(132,213)$. Moreover, the special 3-permutation case $((n+1) \sigma, \sigma_L' n (n+1) \sigma_R')$ contributes $n-1$ elements to $S_{n+1}^2(132,213)$.
    
    \item $\sigma \neq \Id_n$ and $\sigma' = \Id_n$.
    
    This case also contributes $2^{n-1}+n-2$ elements in $S_{n+1}^2(132,213)$. This is a consequence of Corollary \ref{coro} and the reasoning discussed in Case 2. 
    
\end{enumerate}

Now we show that nothing else can contribute to $S_{n+1}^2(132,213)$. Let $\boldsymbol{\sigma} = (\sigma_L n \sigma_R, \sigma_L' n \sigma_R')$ and assume that $\sigma_R$ and $\sigma_R'$ are nonempty and $\sigma \neq \sigma'$. Note that $\sigma$ and $\sigma'$ cannot be $\Id_n.$

Inserting $n+1$ at the beginning of $\sigma$ and $\sigma'$ gives that $\proj(\sigma, \sigma')$ is of the form $(\sigma' \circ \sigma^{-1}) (n+1)$. And since $\sigma \neq \sigma'$, then $\sigma' \circ \sigma^{-1}$ cannot be the identity and contains an occurrence of $21$. Hence $\proj(\sigma, \sigma')$ contains an occurrence of $213$.

We show that $((n+1) \sigma, \sigma_L' n (n+1) \sigma_R')$ cannot avoid $132$ and $213$ either. To see this, let $|\sigma_L| \geq |\sigma_L'|$. Then we evaluate the projection:

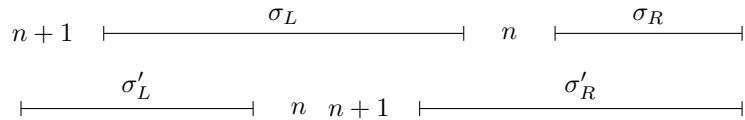
\begin{center}
    \begin{tikzpicture}[coo/.style={coordinate}]
	\path (0,0) coordinate (x1) --++ (1.1,0) coordinate (x2)
			--++ (2,0) coordinate (x3)
			--++ (0.6,0) coordinate (x4)
			--++ (0.8,0) coordinate (x5)
			--++ (0.8,0) coordinate (x6)
			--++ (0.6,0) coordinate (x7)
			--++ (0.6,0) coordinate (x8)
			--++ (0.6,0) coordinate (x9)
			--++ (2.5,0) coordinate (xx);
	\foreach \i in {1,2,3}
		\coordinate[] (L\i) at (0,-\i);

	\draw[|-|] (L1-|x2) -- (L1-|x7) node[above,pos=0.5] {$\sigma_L$};
	\node at (L1-|x8) {$n$};
	\node at (L1-|x1) { \ \ \ \ \ $n+1$};
	\draw[|-|] (L1-|x9) -- (L1-|xx) node[above,pos=0.5] {$\sigma_R$};	
		 
	\draw[|-|] (L2-|x1) -- (L2-|x3) node[above,pos=0.5] {$\sigma_L'$};
	\node at (L2-|x4) {$n$};
	\node at (L2-|x5) {$n+1$};
	\draw[|-|] (L2-|x6) -- (L2-|xx) node[above,pos=0.5] {$\sigma_R'$};

\end{tikzpicture}
\end{center}

This can be represented as $r \pi_L (n+1) \pi_R \ell$, where $\ell$ is an element in $\sigma_L'$ (or $n$ if $\sigma_L'$ is empty), $r$ is an element in $\sigma_R'$, and $\pi_L$ and $\pi_R$ are subpermutations. Since elements in $\sigma_L'$ are greater than elements in $\sigma_R'$, then the projection contains an occurrence of $132$. Figure \ref{fig:132,213 p2} illustrates an example.

    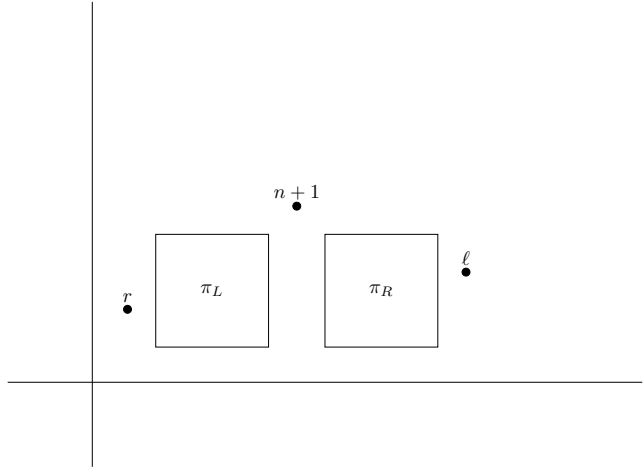
\begin{figure}[h]
    \centering
    \scalebox{0.75}{
    \begin{tikzpicture}
        \foreach [evaluate=\i as \x using int(\i-1)]\i in {0,1,...,11}
        {
            \foreach [evaluate=\j as \y using int(\j-1)] \j in {0,1,...,11}
            {
                \node at (\i,\j)[name=perm-\x-\y,]{};
            }
        }

        \draw ([xshift=-5mm]perm-1--1.south west)--([xshift=-5mm]perm-1-7.north west);
        \draw ([yshift=5mm]perm--1-0.south west)--([yshift=5mm]perm-10-0.south east);

        \filldraw [black] (2,2.67) circle (2pt) node[anchor=south]{$r$};
        
        \draw[draw=black] (2.5,2) rectangle ++(2,2);
        \node at (3.5,3) {$\pi_L$};
        
        \filldraw [black] (5,4.5) circle (2pt) node[anchor=south]{$n+1$};

        \draw[draw=black] (5.5,2) rectangle ++(2,2);
        \node at (6.5,3) {$\pi_R$};

        \filldraw [black] (8,3.33) circle (2pt) node[anchor=south]{$\ell$};

\end{tikzpicture}
}
    \caption{An illustration of $\proj((n+1) \sigma, \sigma_L' n (n+1) \sigma_R')$ when $|\sigma_L| \geq |\sigma_L'|$.}
    \label{fig:132,213 p2}
\end{figure}

Now let $|\sigma_L|<|\sigma_L'|$. To show the $3$-permutation $ (\sigma_1, \sigma_2) = ((n+1) \sigma, \sigma_L' n (n+1) \sigma_R')$ cannot avoid $132$ and $213$, we show that $\proj(\sigma_1, \sigma_2)$ contains either pattern. By Observation \ref{inverselemma}, this is equivalent to showing that $\proj(\sigma_2, \sigma_1)$ contains the pattern $132$ or $213$. We will show that the $3$-permutation $( \sigma_L n (n+1) \sigma_R, (n+1) \sigma')$ cannot avoid $132$ and $213$ when $|\sigma_L| > |\sigma_L'|$. Then we evaluate its projection:

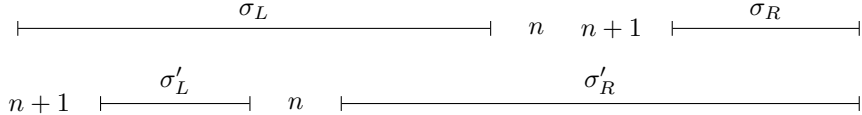
\begin{center}
    \begin{tikzpicture}[coo/.style={coordinate}]
	\path (0,0) coordinate (x1) --++ (1.1,0) coordinate (x2)
			--++ (2,0) coordinate (x3)
			--++ (0.6,0) coordinate (x4)
			--++ (0.6,0) coordinate (x5)
			--++ (2,0) coordinate (x6)
			--++ (0.6,0) coordinate (x7)
			--++ (1,0) coordinate (x8)
			--++ (0.8,0) coordinate (x9)
			--++ (2.5,0) coordinate (xx);
	\foreach \i in {1,2,3}
		\coordinate[] (L\i) at (0,-\i);

	\draw[|-|] (L1-|x1) -- (L1-|x6) node[above,pos=0.5] {$\sigma_L$};
	\node at (L1-|x7) {$n$};
	\node at (L1-|x8) {$n+1$};
	\draw[|-|] (L1-|x9) -- (L1-|xx) node[above,pos=0.5] {$\sigma_R$};	
		 
	\draw[|-|] (L2-|x2) -- (L2-|x3) node[above,pos=0.5] {$\sigma_L'$};
	\node at (L2-|x4) {$n$};
	\node at (L2-|x1) {\ \ \ \ \ $n+1$};
	\draw[|-|] (L2-|x5) -- (L2-|xx) node[above,pos=0.5] {$\sigma_R'$};

\end{tikzpicture}
\end{center}

This is of the form $(\sigma_R' \circ \sigma_R^{-1}) (n+1) \pi_L n \pi_R$, where $\pi_L$ and $\pi_R$ are subpermutations. Since $\sigma_R' \circ \sigma_R^{-1}$ must contain the minimal element $1$, the projection contains an occurrence of $132$. And hence,  $((n+1) \sigma, \sigma_L' n (n+1) \sigma_R')$ cannot avoid $132$ and $213$.

By Corollary \ref{coro}, the $3$-permutation $((n+1) \sigma, \sigma_L' n (n+1) \sigma_R')$ cannot avoid $132$ and $213$ either.

Now we show that for $(\sigma, \sigma') \in S_n^2 \setminus S_n^2(132,213)$, we cannot obtain an element in $ S_{n+1}^2(132,213)$ by inserting the maximal element $n+1$ anywhere in $\sigma$ and $\sigma'$. We will assume that $\sigma$ and $\sigma'$ avoid these patterns but $\sigma' \circ \sigma^{-1}$ does not. Recall that that $n+1$ must be inserted right-adjacent to $n$, except in the cases above, where $n+1$ may be inserted at the beginning of the permutation. So let $\sigma \neq \sigma'$, where neither are the identity. 
We have shown above that $\proj(\sigma_L n (n+1) \sigma_R, \sigma_L' n (n+1) \sigma_R')$ contains $\proj(\sigma, \sigma')$. Now for when $\sigma = \sigma'$ or when one of them is the identity permutation, $\proj(\sigma, \sigma')$ cannot contain $132$ or $213$, a contradiction. Inserting $n+1$ anywhere else will contain an occurrence of $132$ or $213$, and hence, inserting the maximal element $n+1$ anywhere in $\sigma$ and $\sigma'$ cannot produce an element in $ S_{n+1}^2(132,213)$.

Hence, we conclude that \begin{align*}
    a_{n+1} = a_n + 3 \cdot 2^{n-1} + 2(n-1). & \qedhere
\end{align*}

\end{proof}

These theorems allow us to enumerate all $3$-permutations avoiding two patterns of size $3$ that correspond to existing OEIS sequences. Moreover, we have since added the sequence in Theorem \ref{132,213} to the OEIS database \cite{oeis}, allowing the complete classification and enumeration of all $3$-permutations avoiding two patterns of size $3$.

\section{Enumeration of Pattern Avoidance Classes of size 3}

Having enumerated all $3$-permutations avoiding two patterns, we now turn our attention to enumerating $3$-permutations avoiding three patterns, as Simion and Schmidt \cite{simion1985restricted} have done with classic permutations. In Table \ref{triple avoidance}, we extend Bonichon and Morel's \cite{bonichon2022baxter} conjectures to $3$-permutations avoiding three patterns of size $3$.

\begin{table}[h]
    \centering
    \begin{tabular}{|c | c | c | c | c|} 
 \hline
 Patterns & \#TWE & Sequence & OEIS Sequence & Comment \\ [0.5ex] 
 \hline\hline
 $123,132,213$ & 3 &  $1,4,2,0,0, \dots$ & & Terminates after $n=3$ \\ 
 \hline
 $123,132,231$ & 4 &  $1,4,3,0,0,\dots$ & & Terminates after $n=3$ \\
 \hline
 $123,231,312$ & 1 & $1,4,0,0,0, \dots$ & & Terminates after $n=2$ \\
 \hline
 $123,231,321$ & 2 & $1,4,3,0,0, \dots$ & & Terminates after $n=3$ \\
 \hline
 $132,213,312$ & 2 & $1,4,6,8,10, \dots$ & \href{http://oeis.org/A005843}{A005843} & Theorem \ref{132,213,312} \\
 \hline
 $132,213,321$ & 1 & $1,4,9,16,25, \dots$ & \href{http://oeis.org/A000290}{A000290} & Theorem \ref{132,213,321} \\
 \hline
 $132,231,312$ & 2 & $1,4,7,10,13, \dots$  & \href{http://oeis.org/A016777}{A016777} & Theorem \ref{132,231,312} \\
 \hline
 $213, 231, 321$ & 4 & $1,4,6,8,10, \dots$ & \href{http://oeis.org/A005843}{A005843} & Theorem \ref{213,231,321} \\
 \hline
 $231,312,321$ & 1 & $1,4,7,19,40, \dots$ & \href{http://oeis.org/A006130}{A006130} & Theorem \ref{231,312,321} \\
 \hline
\end{tabular}
    \caption{Sequences of $3$-permutations avoiding three permutations of size $3$. The second column indicates the number of trivially Wilf-equivalent classes.}
    \label{triple avoidance}
\end{table}

\begin{theorem}\label{132,213,312}
Let $a_n = |S_n^2(132,213,312)|$. Then $a_{n+1}$ follows the formula $a_{n+1} = 2(n+1)$ for $n>0$ (with initial term $a_1=1$).
\end{theorem}
\begin{proof}
Let $\boldsymbol{\sigma} = (\sigma, \sigma') \in S_n^2(132,213,312)$. Let $\boldsymbol{\sigma}$ be of the form $(\sigma_L n \sigma_R, \sigma_L' n \sigma_R').$ 

Note that $\sigma_R$ and $\sigma_R'$ have to be either empty or consecutively decreasing, and similarly, $\sigma_L$ and $\sigma_L'$ have to be either empty or consecutively increasing. Moreover, every element in $\sigma_L$ and $\sigma_L'$ must be larger than every element in $\sigma_R$ and $\sigma_R'$, respectively. If not, there would be an occurrence of $132$. 

We have the following cases:
\begin{enumerate}
    \item $\sigma_L$, $\sigma_R$ are nonempty.
    
    If $\sigma_L'$ and $\sigma_R'$ are nonempty, consider $\boldsymbol{\sigma} = (\sigma_L n \sigma_R, \sigma_L' n \sigma_R')$ in $S_{n}^2 (132,213,312)$. Note that inserting $n+1$ right-adjacent to $n$ in both $\sigma$ and $\sigma'$ will avoid $132$, $213$, and $312$. In particular, $(\sigma_L n (n+1) \sigma_R, \sigma_L' n (n+1) \sigma_R')$ avoids $132$, $213$, and $312$. 
    
    Now we show that inserting $n+1$ anywhere else in $\boldsymbol{\sigma}$ does not yield an element of $S_{n+1}^2 (132,213,312)$. Consider $\sigma_L n \sigma_R$. We cannot insert $n+1$ in the beginning of this permutation, or else there would be an instance of $312$. Further, we cannot insert $n+1$ anywhere to the left of $n$, or else there would be an instance of $132$. There would also be an occurrence of $213$ if $n+1$ is inserted anywhere to the right of $n$ that is not adjacent to $n$. 
    
    Hence, $n+1$ is forced to be right-adjacent to $n$ in $\sigma$. The same conclusion follows for $\sigma'$.
    
    If $\sigma_L'$ is empty, then $\sigma' = \rev(\Id_n)$. The projection $\sigma' \circ \sigma^{-1}$ contains an occurrence of $132$, and therefore, this case is impossible. Similarly, if $\sigma_R'$ is empty, then $\sigma' = \Id_n$. Note that the projection $\sigma' \circ \sigma^{-1}$ contains an occurrence of $312$ since $\sigma$ contains an occurrence of $231$, and this case is also impossible.
    
    Therefore, every element in this case contributes $1$ element in $S_{n}^2 (132,213,312)$.
    
    \item $\sigma_L$ is empty.
    
    If both $\sigma_L'$ and $\sigma_R'$ are nonempty, then $\boldsymbol{\sigma} = (\rev(\Id_n), \sigma_L' n \sigma_R')$. Taking the projection $\proj(\rev(\Id_n), \sigma_L' n \sigma_R')$ gives an instance of $132$ because every element in $\sigma_L'$ is larger than every element in $\sigma_R'$. Thus, $\boldsymbol{\sigma}$ is not a valid element in $S_{n}^2 (132,213,312)$, and we conclude this case is impossible.
    
    If $\sigma_L'$ is empty, then we conclude that $(\sigma, \sigma') = (\rev(\Id_n), \rev(\Id_n))$. Following similar logic to the previous case, $n+1$ must be inserted adjacent to $n$ in both $\sigma$ and $\sigma'$ to avoid $312$ and $213$. Note that $((n+1)\rev(\Id_n), (n+1)\rev(\Id_n))$ avoids $132$, $213$, and $312$. However, $\proj((n+1)\rev(\Id_n), (n (n+1) \rev(\Id_{n-1}))$ and $\proj(n(n+1)\rev(\Id_{n-1}), (n+1)\rev(\Id_n))$  cannot avoid $132$, $213$, and $312$. Therefore, $(\rev(\Id_n), \rev(\Id_n))$ contributes one more element to $S_{n+1}^2 (132,213,312)$, in addition to the one obtained by inserting $n+1$ right-adjacent to $n$ in both $\sigma$ and $\sigma'$ as discussed in the previous case.
    
    If $\sigma_R'$ is empty, then $(\sigma, \sigma') = (\rev(\Id_n), \Id_n)$. The element $n+1$ must be inserted adjacent to $n$ in $\sigma$, and $n+1$ must be inserted at the end of $\sigma'$. We can see that the $3$-permutation $((n+1)\rev(\Id_n), \Id_n (n+1))$ is an element of $S_{n+1}^2 (132,213,312)$ and furthermore, the $3$-permutation $(n (n+1) \rev(\Id_{n-1}), \Id_n (n+1))$ is not an element, because its projection $\sigma' \circ \sigma^{-1}$ contains an instance of $312$. 
    
    Hence, each element in this case contributes $1$ element to $S_{n+1}^2 (132,213,312)$, with the exception of $(\rev(\Id_n), \rev(\Id_n))$, which contributes $2$ elements to $S_{n+1}^2 (132,213,312)$.
    
    \item $\sigma_R$ is empty. 
    
    If $\sigma_L'$ is nonempty, then $\boldsymbol{\sigma} = (\Id_n, \sigma_L' n \sigma_R')$. Then $n+1$ is forced to be right-adjacent to $n$ for both $\sigma$ and $\sigma'$, which contributes $1$ element to $S_{n+1}^2 (132,213,312)$.
    
    If $\sigma_L'$ is empty, then $(\sigma, \sigma') = (\Id_n, \rev(\Id_n))$. Note that $(\Id_n(n+1), (n+1)\rev(\Id_n))$ avoids $132$, $213$, and $312$. Hence, the $3$-permutation $(\Id_n, \rev(\Id_n))$ contributes one more element to $S_{n+1}^2 (132,213,312)$ in addition to the one obtained by inserting $n+1$ right-adjacent to $n$ in  $\sigma$ and $\sigma'$.
    
    Hence, each element in this case contributes $1$ element to $S_{n+1}^2 (132,213,312)$, with the exception of $(\Id_n, \rev(\Id_n))$, which contributes $2$ elements to $S_{n+1}^2 (132,213,312)$.

\end{enumerate}

Inserting $n+1$ anywhere else in $\boldsymbol{\sigma}$ cannot provide an element in $S_{n+1}^2 (132,213,312)$.

We show that for $(\sigma, \sigma') \in S_n^2 \setminus S_n^2(132,213,312)$, we cannot obtain an element in $ S_{n+1}^2(132,213,312)$ by inserting the maximal element $n+1$ anywhere in $\sigma$ and $\sigma'$. Let $\sigma$ and $\sigma'$ avoid these patterns but let $\sigma' \circ \sigma^{-1}$ contain them. It is enough to check the cases above:

\begin{enumerate}
    \item $\sigma_L$, $\sigma_R$ are nonempty.
    
    It is straightforward to check that $\proj(\sigma_L n (n+1) \sigma_R, \sigma_L' n (n+1) \sigma_R')$ contains $\sigma' \circ \sigma^{-1}$ and hence, contains an occurrence of $132$, $213$, or $312$. The proof in Case 1 above shows that the maximal element $n+1$ must be inserted right-adjacent to $n$ to avoid an occurrence of these patterns in $\sigma$ and $\sigma'$.

    \item $\sigma_L$ is empty.
    
    Then $(\sigma, \sigma') = (\rev(\Id_n), \sigma_L' n \sigma_R')$. It is also straightforward to check that when the maximal element $n+1$ is inserted adjacent to $n$ in both $\sigma$ and $\sigma'$, the projection of the resulting $3$-permutation contains an occurrence of $132$ if $\sigma_R'$ and $\sigma_L'$ are nonempty. If either $\sigma_R'$ or $\sigma_L'$ are empty, then it is impossible for $\sigma' \circ \sigma^{-1}$ to contain instances of $132$, $213$, $312$, a contradiction. 

    \item $\sigma_R$ is empty. 
    
    Then $(\sigma, \sigma') = (\Id_n, \sigma_L' n \sigma_R')$. The projection of this $3$-permutation is $\sigma'$, and it is impossible for the projection to contain an occurrence of these patterns while $\sigma'$ avoids them. 
\end{enumerate}

Therefore, we have shown $a_{n+1} = a_n +2$. We have the base case $a_2 = 4$, and we then have \begin{align*}
    a_{n+1} = 2(n+1). & \qedhere
\end{align*}
\end{proof}

\begin{theorem}\label{132,213,321}
Let $a_n = |S_n^2(132,213,321)|$. Then $a_{n+1}$ follows the formula $a_{n+1} = (n+1)^2$.
\end{theorem}
\begin{proof}
Let $\boldsymbol{\sigma} = (\sigma, \sigma') \in S_n^2(132,213,321)$. Write $(\sigma, \sigma')$ as $(\sigma_L n \sigma_R, \sigma_L' n \sigma_R')$. 

Using a similar reasoning discussed in Theorem \ref{132,213,312}, note that $\sigma_L$, $\sigma_L'$, $\sigma_R$, and $\sigma_R'$ are consecutively increasing. Further, every element in $\sigma_L$ and $\sigma_L'$ is larger than every element in $\sigma_R$ and $\sigma_R'$, respectively.

Using the reasoning in Theorem \ref{132,213,312}, $(\sigma_L n (n+1) \sigma_R, \sigma_L' n (n+1) \sigma_R')$ is in $S_{n+1}^2(132,213,321)$. This contributes $a_n$ different $3$-permutations to $S_{n+1}^2(132,213,321)$. We also have the following cases:

\begin{enumerate}
    \item $\sigma_R$ is empty and $\sigma_R'$ is nonempty.
    
    Note that this implies that $\sigma = \Id_n$ and $\sigma' \neq \Id_n$. Then $((n+1) \Id_n, \sigma_L' n (n+1) \sigma_R')$ avoids $132$, $213$, and $321$.
    
    Inserting $n+1$ anywhere else in $\boldsymbol{\sigma}$ cannot avoid $132$, $213$, and $321$. We must insert $n+1$ right-adjacent to $n$ in $\sigma'$. If $n+1$ is left of $n$, then $\sigma'$ contains an instance of $321$. If $n+1$ is right of $n$ but not adjacent, then $\sigma'$ contains an instance of $213$. 
    
    In $\sigma$, we must either insert $n+1$ at the beginning of the permutation or the end of the permutation. However, inserting $n+1$ at the end of the permutation would correspond to a $3$-permutation we have already considered above.
    
    And hence, $((n+1) \Id_n, \sigma_L' n (n+1) \sigma_R')$ is the only $3$-permutation we can construct in $S_{n+1}^2(132,213,321)$. This case contributes $n-1$ elements to $S_{n+1}^2(132,213,321)$.
    
    \item $\sigma_R$ is nonempty and $\sigma_R'$ is empty.
    
    This implies that $\sigma' = \Id_n$ and $\sigma \neq \Id_n$. Note that $(\sigma_L n (n+1) \sigma_R, (n+1) \Id_n)$ belongs to $S_{n+1}^2(132,213,321)$. Using a similar argument as in Case 1, inserting $n+1$ in $\boldsymbol{\sigma}$ anywhere else does not avoid $132$, $213$, and $312$. Hence, this case contributes $n-1$ different $3$-permutations to $S_{n+1}^2(132,213,321)$.
    
    \item Both $\sigma_R$, $\sigma_R'$ are empty.
    
    This implies that $\sigma = \sigma' = \Id_n$. We see that $((n+1)\Id_n,\Id_n (n+1))$, $((n+1)\Id_n, (n+1)\Id_n)$, and $(\Id_n (n+1), (n+1) \Id_n)$ all avoid $132$, $213$, and $321$. And the same reasoning as in Case 1 shows that inserting $n+1$ anywhere else in $\boldsymbol{\sigma}$ cannot avoid these patterns, and this case contributes $3$ elements to $S_{n+1}^2(132,213,321)$.
\end{enumerate}

Finally, when $\sigma$ and $\sigma'$ are not the identity permutation, then $\sigma_R$ and $\sigma_R'$ are both nonempty, and the same argument in Case 1 shows that inserting $n+1$ anywhere not right-adjacent to $n$ in $\sigma$ and $\sigma'$ cannot avoid $132$, $213$, and $321$. Hence, no other insertions of $n+1$ in $\boldsymbol{\sigma}$ produce an element in $S_{n+1}^2(132,213,321)$. 

We show that for $(\sigma,\sigma') \in S_n^2 \setminus S_{n}^2(132,213,321)$, we cannot obtain an element in $ S_{n+1}^2(132,213,321)$ by inserting the maximal element $n+1$ anywhere in $\sigma$ and $\sigma'$. We assume that $\sigma$ and $\sigma'$ avoid these patterns but $\sigma' \circ \sigma^{-1}$ does not. Write $(\sigma, \sigma') = (\sigma_L n \sigma_R, \sigma_L' n \sigma_R')$. It is straightforward to check that $\proj(\sigma_L n (n+1) \sigma_R, \sigma_L' n (n+1) \sigma_R')$ contains $\sigma' \circ \sigma^{-1}$. Further, we have the same special cases as above:

\begin{enumerate}
    \item $\sigma_R$ is empty and $\sigma_R'$ is nonempty.
    
    Then the projection $\proj(\Id_n, \sigma_L' n \sigma_R')$ is $\sigma_L' n \sigma_R'$, and hence it is impossible for the projection to contain these patterns while $\sigma'$ avoids them.

    \item $\sigma_R$ is nonempty and $\sigma_R'$ is empty.
    
    The projection $\proj(\sigma_L n \sigma_R, \Id_n)$ is of the form $\pi_L \pi_R$, where $\pi_L$ and $\pi_R$ are both consecutively increasing and each element of $\pi_L$ is greater than each element of $\pi_R$. A permutation of this form cannot contain $132$, $213$, or $321$, and it is impossible for $\proj(\sigma_L n \sigma_R, \Id_n)$ to contain these patterns while $\sigma$ avoids them.

    \item Both $\sigma_R$, $\sigma_R'$ are empty.
    
    It is impossible for $\proj(\Id_n, \Id_n)$ to contain these patterns.
\end{enumerate}

Hence, we have $$a_{n+1} = a_n + 2n + 1.$$

The base case is $a_1 = 1$, and we conclude that \begin{align*}
    a_{n+1} = (n+1)^2. & \qedhere
\end{align*}
\end{proof}

\begin{theorem}\label{132,231,312}
Let $a_n = |S_n^2(132,231,312)|$. Then $a_n$ satisfies the recurrence $a_{n+1} = a_n+3$ with initial term $a_1 = 1$.
\end{theorem}
\begin{proof}
Let $\boldsymbol{\sigma} = (\sigma, \sigma') \in S_n^2(132,231,312)$. Write $(\sigma, \sigma')$ as $(\sigma_L n \sigma_R, \sigma_L' n \sigma_R')$. Note that $n \sigma_R$ and $n \sigma_R'$ must be consecutively decreasing to avoid $312$ and $132$. 

We insert the maximal element $n+1$ into $\sigma$ and $\sigma'$ to count how many elements in $S_{n+1}^2(132,231,312)$ there are. Note that $(\sigma (n+1), \sigma' (n+1))$ avoids $132$, $231$, and $312$. This contributes $a_n$ different $3$-permutations to $S_{n+1}^2(132,231,312)$. We have the following additional cases:

\begin{enumerate}
    \item $\sigma = \rev(\Id_n)$ and $\sigma' \neq \rev(\Id_n)$.
    
    This forces $\sigma'$ to be the identity. Then $((n+1)\rev(\Id_n), \sigma' (n+1))$ avoids $132$, $231$, and $312$. Now inserting $n+1$ anywhere else in $\boldsymbol{\sigma}$ cannot avoid these patterns. Namely, if $n+1$ is inserted anywhere not in the beginning or end of $\sigma$, there is an occurrence of $231$. Moreover, inserting $n+1$ into the beginning of $\sigma'$ contains $312$. If $n+1$ is inserted anywhere not in the beginning or end of $\sigma'$, there is an occurrence of $132$ in $\sigma'$. Hence, $((n+1)\rev(\Id_n), \sigma' (n+1))$ is the only element we can construct in $S_{n+1}^2(132,231,312)$ in this case. And this case contributes one element to $S_{n+1}^2(132,231,312)$.
    
    \item $\sigma' = \rev(\Id_n)$ and $\sigma \neq \rev(\Id_n)$.
    
    Then similar to Case 1, $(\sigma (n+1), (n+1) \rev(\Id_n))$ avoids $132$, $231$, and $312$, and inserting $n+1$ anywhere else into this $3$-permutation cannot result in a $3$-permutation in $S_{n+1}^2(132,231,312)$. Hence, this case contributes one element to $S_{n+1}^2(132,231,312)$.
    
    \item $\sigma = \sigma' =\rev(\Id_n)$.
    
    Note that $((n+1) \rev(\Id_n), (n+1) \rev(\Id_n))$ avoids $132$, $231$, and $312$. Now we show that no other insertions of $n+1$ into this $3$-permutation results in a $3$-permutation that avoids these patterns. Note that $\proj( \rev(\Id_n) (n+1), (n+1) \rev(\Id_n))$ contains an occurrence of $231$ and further, note that $\proj((n+1) \rev(\Id_n), \rev(\Id_n) (n+1))$ contains an occurrence of $312$. Therefore, this case contributes one element to $S_{n+1}^2(132,231,312)$.
\end{enumerate}

Inserting $n+1$ into $(\sigma, \sigma') = (\sigma_L n \sigma_R, \sigma_L' n \sigma_R')$ anywhere else cannot avoid $132$, $231$, and $312$, where $\sigma, \sigma' \neq \rev(\Id_n)$. More specifically, inserting $n+1$ left-adjacent to $n$ contains $132$ and inserting $n+1$ anywhere to the left of this contains $312$. Further, inserting $n+1$ anywhere to the right of $n$ (but not at the end of the permutation) contains $231$. Hence, we must insert $n+1$ at the end of the permutation, and no other insertions of $n+1$ in $\boldsymbol{\sigma}$ avoid $132$, $231$, and $312$.

Now for $(\sigma,\sigma') \in S_n^2 \setminus S_{n}^2(132,231,312)$, we cannot obtain an element in $ S_{n+1}^2(132,231,312)$ by inserting the maximal element $n+1$ anywhere in $\sigma$ and $\sigma'$. This follows a similar argument to the one presented in Theorem \ref{132,213,321}.

Thus, we conclude that $$a_{n+1} = a_n+3.$$

Since our base case is $a_1 = 1$, this is equivalent to $a_{n+1} = 3n+1$.
\end{proof}

\begin{theorem}\label{213,231,321}
Let $a_n = |S_n^2(213, 231, 321)|$. Then $a_{n+1}$ follows the formula $a_{n+1} = 2(n+1)$ for $n>0$ (with initial term $a_1 = 1$).
\end{theorem}
\begin{proof}
Let $\boldsymbol{\sigma} = (\sigma, \sigma') \in S_n^2(213,231,321)$. Writing $(\sigma, \sigma')$ as $(\sigma_L n \sigma_R, \sigma_L' n \sigma_R')$, note that $\sigma_L$, $\sigma_L'$, $\sigma_R$, and $\sigma_R'$ are all consecutively increasing or empty.

We insert the maximal element $n+1$ to $\sigma$ and $\sigma'$ in an attempt to construct an element in $S_{n+1}^2(213,231,321)$. If $\sigma_R$ and $\sigma_R'$ are nonempty, we cannot construct an element of $S_{n+1}^2(213,231,321)$ via insertion because inserting $n+1$ to the left of $n$ contains $321$, inserting $n+1$ right-adjacent to $n$ contains $231$, and inserting $n+1$ anywhere else contains $213$. Then it is enough to consider $(\sigma, \sigma') = (\Id_n, \Id_n)$. We have two cases:

\begin{enumerate}
    \item We insert $n+1$ to the end of $\sigma$. 
    
    Then we can insert $n+1$ anywhere in $\sigma'$ and the resulting $3$-permutation is an element of $S_{n+1}^2(213,231,321)$. This case contributes $n+1$ different elements to $S_{n+1}^2(213,231,321)$.
    
    \item We do not insert $n+1$ to the end of $\sigma$. 
    
    Note that inserting $n+1$ into the same position in $\sigma$ and $\sigma'$ avoids $213$, $231$, and $321$. Further, $(\Id_{n-1} (n+1)n, \Id_{n} (n+1))$ also avoids these patterns. 
    
    Inserting $n+1$ anywhere else contains one of these patterns because the resulting projection contains either $321$ or $231$. Hence, this case contributes $n+1$ different $3$-permutations to $S_{n+1}^2(213,231,321)$.
\end{enumerate}

Now for $(\sigma,\sigma') \in S_n^2 \setminus S_{n}^2(213,231,312)$, we cannot obtain an element in $ S_{n+1}^2(213,231,312)$ by inserting the maximal element $n+1$ anywhere in $\sigma$ and $\sigma'$. It is enough to consider $(\sigma, \sigma') = (\Id_n, \Id_n)$. Note that it is impossible for $\proj(\sigma, \sigma')$ to contain $213$, $231$, or $321$ while $\sigma$ and $\sigma'$ avoid these patterns, and it follows that \begin{align*}
    a_{n+1} = 2(n+1). & \qedhere
\end{align*}
\end{proof}

\begin{theorem}\label{231,312,321}
Let $a_n = |S_n^2(231,312,321)|$. Then $a_n$ satisfies the recurrence $$a_{n+1} = a_n + 3a_{n-1}$$ with initial terms $a_1 = 1$ and $a_2 = 4$.
\end{theorem}
\begin{proof}
Let $\boldsymbol{\sigma} = (\sigma, \sigma') \in S_n^2(231,312,321)$. Write $(\sigma, \sigma')$ as $(\sigma_L n \sigma_R, \sigma_L' n \sigma_R')$.

Note that $(\sigma (n+1), \sigma' (n+1))$ is an element of $S_{n+1}^2(231,312,321)$. This contributes $a_n$ different $3$-permutations to $S_{n+1}^2(231,312,321)$. We consider the following additional case: when $\sigma_R$ and~$\sigma_R'$ are empty.

Then $(\sigma, \sigma') = (\sigma_L n, \sigma_L' n)$, and thus, $(\sigma_L (n+1)n, \sigma_L' n(n+1))$, $(\sigma_L (n+1)n, \sigma_L' (n+1)n)$, and $(\sigma_L n(n+1), \sigma_L' (n+1)n)$ are all elements in $S_{n+1}^2(231,312,321)$. Inserting $n+1$ anywhere else contains an occurrence of $312$. Thus, this case contributes $3a_{n-1}$ distinct $3$-permutations to $S_{n+1}^2(231,312,321)$.

Now when either $\sigma_R$ and $\sigma_R'$ are nonempty, we show that inserting $n+1$ anywhere but the end of the $3$-permutation cannot avoid $231$, $312$, and $321$. Let $\sigma_R$ be nonempty. Then we must insert $n+1$ at the end of $\sigma$; otherwise, inserting $n+1$ to the right of $n$ contains $231$, inserting left-adjacent to $n$ contains $321$, and inserting to the left of $n$ contains $312$. And we evaluate the projection of $(\sigma_L n \sigma_R (n+1), \sigma_L' (n+1) n)$:

\begin{center}
    \begin{tikzpicture}[coo/.style={coordinate}]
	\path (0,0) coordinate (x1) --++ (2,0) coordinate (x2)
			--++ (0.6,0) coordinate (x3)
			--++ (0.6,0) coordinate (x4)
			--++ (0.6,0) coordinate (x5)
			--++ (0.8,0) coordinate (x6)
			--++ (0.8,0) coordinate (x7);
	\foreach \i in {1,2,3}
		\coordinate[] (L\i) at (0,-\i);

	\draw[|-|] (L1-|x1) -- (L1-|x2) node[above,pos=0.5] {$\sigma_L$};
	\node at (L1-|x3) {$n$};
	\node at (L1-|x7) {$n+1$};
	\draw[|-|] (L1-|x4) -- (L1-|x6) node[above,pos=0.5] {$\sigma_R$}; 
	
	\draw[|-|] (L2-|x1) -- (L2-|x5) node[above,pos=0.5] {$\sigma_L'$};
	\node at (L2-|x7) {$n$};
	\node at (L2-|x6) {$n+1$};
		
\end{tikzpicture}
\end{center}

Since $\sigma_R$ is nonempty, this contains an instance of $312$. The case where $\sigma_R'$ is nonempty is similar. Inserting $n+1$ anywhere else in $\boldsymbol{\sigma}$ cannot produce an element in $S_{n+1}^2(231,312,321)$.

Now for $(\sigma,\sigma') \in S_n^2 \setminus S_{n}^2(231,312,321)$, we cannot obtain an element in $ S_{n+1}^2(231,312,321)$ by inserting the maximal element $n+1$ anywhere in $\sigma$ and $\sigma'$. It is straightforward to check that the projections of $(\sigma (n+1), \sigma' (n+1))$, $(\sigma_L (n+1)n, \sigma_L' n(n+1))$, $(\sigma_L (n+1)n, \sigma_L' (n+1)n)$, and $(\sigma_L n(n+1), \sigma_L' (n+1)n)$ contain instances of $\sigma' \circ \sigma^{-1}$ and hence, contain instances of $231$, $312$, or $321$. The proof above shows that inserting $n+1$ anywhere else in $\sigma$ and $\sigma'$ cannot avoid these patterns.

Thus, it follows that \begin{align*}
    a_{n+1} = a_n + 3a_{n-1}. & \qedhere
\end{align*}
\end{proof}

\section{Final Remarks and Open Problems}

In this paper, we completely enumerated $3$-permutations avoiding two patterns of size $3$ and three patterns of size $3$. The theorems in this paper prove all the conjectures by Bonichon and Morel \cite{bonichon2022baxter} regarding $3$-permutations avoiding two patterns of size $3$ and extend their conjectures to classify  $3$-permutations avoiding all classes of three patterns of size $3$. We conclude with the following open problem.

\begin{problem}
Find combinatorial bijections to explain the relationships between the $3$-permutation avoidance classes found in this paper and their recurrence relations.
\end{problem}

For example, we notice that the sequence \href{http://oeis.org/A001787}{A001787} in Theorem \ref{132,231} counts the number of $132$-avoiding permutations of length $n+2$ with exactly one occurrence of a $123$-pattern and the number of Dyck $(n+2)$-paths with exactly one valley at height $1$ and no higher valley \cite{oeis}. In general, the problem of enumerating $d$-permutations avoiding sets of small patterns is widely open. Since several of these enumeration sequences correspond to sequences on the OEIS database \cite{oeis}, there are certainly interesting combinatorial properties of these $3$-permutation avoidance classes, and there are several bijections to find that explain these sequences.

\section*{Acknowledgements}

This research was conducted at the 2022 University of Minnesota Duluth REU and is supported by Jane Street Capital, the NSA (grant number H98230-22-1-0015), the NSF (grant number DMS-2052036), and the Harvard College Research Program. The author thanks the anonymous referees for their helpful feedback and suggestions. He is also indebted to Joe Gallian for his dedication and organization of the University of Minnesota Duluth REU. Lastly, a special thanks to Joe Gallian, Amanda Burcroff, Maya Sankar, and Andrew Kwon for their invaluable advice on this paper.

\section*{Conflict of Interest Statement}
The author states that there is no conflict of interest. The manuscript has no associated data.

\newpage
\bibliography{references}
\bibliographystyle{plain}

$\\ \\ \\$
\end{document}